\documentclass[smallcondensed]{svjour3}

\usepackage[english]{babel}
\usepackage[latin1]{inputenc}
\usepackage[T1]{fontenc}
\usepackage{amsmath}
\usepackage{graphicx}
\usepackage{amssymb}
\usepackage{comment}
\usepackage{amsfonts}
\usepackage{mathcomp}
\usepackage{parskip}
\usepackage{boxedminipage}
\usepackage{setspace}
\usepackage{enumerate}
\usepackage{hyperref}
\usepackage{ulem}
\usepackage{stmaryrd}
\usepackage{amsthm}
\usepackage{subfig}
\usepackage{bm}
\usepackage{tikz}
\usetikzlibrary{patterns}
\usetikzlibrary{decorations.markings}
\usetikzlibrary{arrows}
\usepackage{anyfontsize}
\usepackage[a4paper]{geometry}

\newcommand{\mat}[1]{\ensuremath{\mathsf{#1}}}
\newcommand{\ip}[3]{\left\langle {#1} , {#3} \right\rangle_{\!#2}}
\newcommand{\ipL}[2]{\ip{#1}{N_L}{#2}}
\newcommand{\ipLi}[2]{\ip{#1}{N_{L_i}}{#2}}
\newcommand{\ipR}[2]{\ip{#1}{N_R}{#2}}
\newcommand{\ipC}[2]{\ip{#1}{N}{#2}}
\newcommand{\Ent}{S}
\newcommand{\dEnt}{\Ent_t}
\newcommand{\dEntij}{\left(\Ent_{t}\right)_{ij}}
\newcommand{\TotCon}{IU}
\newcommand{\dTotCon}{\TotCon_t}
\newcommand{\TotEnt}{IS}
\newcommand{\dTotEnt}{\TotEnt_t}
\newcommand{\ICon}{\dTotCon^I}
\newcommand{\IEnt}{\dTotEnt^I}
\newcommand{\eqInd}{q}
\newcommand{\FstarLRK}{\tilde{\mat{F}}_{L,R}^{\EC,\eqInd}}
\newcommand{\FstarLRKT}{\left(\tilde{\mat{F}}_{L,R}^{\EC,\eqInd}\right)^{\!T}}

\newcommand{\FstarLiRK}{\tilde{\mat{F}}_{L_i,R}^{\EC,\eqInd}}
\newcommand{\FstarLiRKT}{\left(\tilde{\mat{F}}_{L_i,R}^{\EC,\eqInd}\right)^{\!T}}

\newcommand{\evL}{\bm{V}^{\eqInd,L}}

\newcommand{\evR}{\bm{V}^{\eqInd,R}}
\newcommand{\evRT}{\left(\bm{V}^{\eqInd,R}\right)^{\!T}\!}
\newcommand{\One}{\bm{1}}

\newcommand{\OneR}{\One^{R}}
\newcommand{\OneRT}{\left(\!\One^{R}\right)^{\!T}\!}
\newcommand{\OneL}{\One^{L}}
\newcommand{\OneLT}{\left(\!\One^{L}\right)^{\!T}\!}
\newcommand{\OneLi}{\One^{L_i}}
\newcommand{\OneLiT}{\left(\!\One^{L_i}\right)^{\!T}\!}

\newcommand{\EC}{\mathrm{EC}}
\newcommand{\ES}{\mathrm{ES}}
\newcommand{\uk}{\bm{U}}
\newcommand{\ukij}{\uk_{ij}}
\newcommand{\ukim}{\uk_{im}}
\newcommand{\ukmj}{\uk_{mj}}
\newcommand{\fsharp}{\tilde{\bm{F}}^{\EC}}
\newcommand{\gsharp}{\tilde{\bm{G}}^{\EC}}
\newcommand{\fstar}{\tilde{\bm{F}}^{\EC}}
\newcommand{\gstar}{\tilde{\bm{G}}^{\EC}}

\newcommand{\mass}{\mat M}
\newcommand{\massL}{\mat M_L}
\newcommand{\massLi}{\mat M_{L_i}}
\newcommand{\massR}{\mat M_R}

\newcommand{\D}{\mat D}

\newcommand{\PLtoX}{\mat{P}_{L2\Xi}}
\newcommand{\PXtoL}{\mat{P}_{\Xi2L}}
\newcommand{\PRtoX}{\mat{P}_{R2\Xi}}
\newcommand{\PXtoR}{\mat{P}_{\Xi2R}}

\newcommand{\dxdy}{\,\mathrm{d}x\mathrm{d}y}
\newcommand{\dxideta}{\,\mathrm{d}\xi\mathrm{d}\eta}

\newcommand{\jump}[1]{\left\llbracket #1 \right\rrbracket}

\theoremstyle{theorem}
\newtheorem{thm}{Theorem}
\newtheorem{lem}{Lemma}
\newtheorem{cor}{Corollary}
\theoremstyle{remark}
\newtheorem{rem}{Remark}

\usepackage{mathtools}
\DeclareMathOperator{\diag}{diag}
\newcommand{\EE}{\mathbb{E}}

\numberwithin{equation}{section}

\begin{document}

\title{An Entropy Stable $h/p$ Non-Conforming Discontinuous Galerkin Method with the Summation-by-Parts Property}
\titlerunning{An Entropy Stable $h/p$ Non-Conforming DG Method with the SBP Property}
\author{Lucas Friedrich \and Andrew R. Winters \and David C.~Del Rey Fern\'{a}ndez \and Gregor J. Gassner \and Matteo Parsani \and Mark H. Carpenter}
\institute{Lucas Friedrich (\email{lfriedri@math.uni-koeln.de}) \and Andrew R. Winters \and Gregor J. Gassner \at Mathematical Institute, University of Cologne, Cologne, Germany \\ David C.~Del Rey Fern\'{a}ndez \at National Institute of Aerospace and Computational AeroSciences Branch, NASA Langley Research Center, Hampton, VA, USA \\ Mark H. Carpenter \at Computational AeroSciences Branch, NASA Langley Research Center, Hampton, VA, USA \\ Matteo Parsani \at King Abdullah University of Science and Technology (KAUST), Computer Electrical and Mathematical Science and Engineering Division (CEMSE), Extreme Computing Research Center (ECRC), Thuwal, Saudi Arabia}

\date{Received: date / Accepted: date}

\maketitle

\begin{abstract}
This work presents an entropy stable discontinuous Galerkin (DG) spectral element approximation for systems of non-linear conservation laws with general geometric $(h)$ and polynomial order $(p)$ non-conforming rectangular meshes. The crux of the proofs presented is that the nodal DG method is constructed with the collocated Legendre-Gauss-Lobatto nodes. This choice ensures that the derivative/mass matrix pair is a summation-by-parts (SBP) operator such that entropy stability proofs from the continuous analysis are discretely mimicked. Special attention is given to the coupling between non-conforming elements as we demonstrate that the standard mortar approach for DG methods does not guarantee entropy stability for non-linear problems, which can lead to instabilities. As such, we describe a precise procedure and modify the mortar method to guarantee entropy stability for general non-linear hyperbolic systems on $h/p$ non-conforming meshes. We verify the high-order accuracy and the entropy conservation/stability of fully non-conforming approximation with numerical examples.
\end{abstract}

\keywords{Summation-by-Parts \and Discontinuous Galerkin \and Entropy Conservation \and Entropy Stability \and $h/p$ Non-Conforming Mesh \and Non-Linear Hyperbolic Conservation Laws}

\section{Introduction}

The non-conforming discontinuous Galerkin spectral element method (DGSEM), with respect to either mesh refinement introducing hanging nodes ($h$), varying the polynomial order ($p$) across elements or both ($h/p$), is attractive for problems with strong varying feature sizes across the computational domain because the number of degrees of freedom can be significantly reduced. Past work has demonstrated that the mortar method \cite{Kopriva1996b,Kopriva2002} is a projection based approach to construct the numerical flux at non-conforming element interfaces. The mortar approach retains high-order accuracy as well as the desirable excellent parallel computing properties of the DGSEM \cite{Tan2012,Hindenlang201286}. However, we are in particular interested in building a high order DG scheme with the aforementioned positive properties that is provably entropy stable for general non-linear problems. That is, the non-conforming DGSEM should satisfy the second law of thermodynamics discretely. Our interest is twofold:
\begin{enumerate}
\item The numerical approximation will obey one of the most fundamental physical laws.
\item For under-resolved flow configurations, like turbulence, entropy stable approximations have been shown to be robust, e.g. \cite{Bohm2017,Gassner2016,Yee2017,Winters2017}.
\end{enumerate}
The subject of non-conforming approximations is natural in the context of applications that contain a wide variety of spatial scales. This is because non-conforming methods can focus the degrees of freedom in a discretization where they are needed. There is some work available for entropy stable $p$ non-conforming DG methods applied to the compressible Navier-Stokes equations, e.g. Parsani et al. \cite{Parsani2016,Parsani2015b} or Carpenter et al. \cite{Carpenter2016}. 

This work presents an extension of entropy stable non-conforming DG methods to include the hanging nodes ($h$) and the combination of varying polynomials and hanging mesh nodes ($h/p$) for general non-linear systems of conservation laws. We demonstrate that the derivative matrix in the DG context must satisfy the summation-by-parts (SBP) property as well as how to modify the mortar method \cite{Kopriva1996b} to guarantee high-order accuracy and entropy stability on rectangular meshes. As the algorithm of the method is still similar to the mortar approach, parallel scaling efficiency is not influenced by the modifications.

We begin with a short overview of the different DG approaches on rectangular meshes. First, we provide a background of the non-linear entropy stable DGSEM on conforming quadrilateral meshes. We then introduce the popular mortar approach in the nodal DG context. However, we demonstrate that this well-known non-conforming coupling is insufficient to guarantee entropy stability for non-linear partial differential equations (PDEs). The main result of this work is to marry these two powerful approaches, i.e., entropy stability of conforming DG methods and non-conforming coupling, to create a novel, entropy stable, high-order, non-conforming DGSEM for non-linear systems of conservation laws.

\subsection{Entropy Stable Conforming DGSEM}

We consider systems of non-linear hyperbolic conservation laws in a two dimensional spatial domain $\Omega\subset\mathbb{R}^2$ with $t\in\mathbb{R}^+$
\begin{equation}\label{eq:2DconsLaw}
\bm{u}_t + \bm{f}_x\!\left(\bm{u}\right) + \bm{g}_y\!\left(\bm{u}\right) = \bm{0},
\end{equation}
with suitable initial and boundary conditions. The extension to a three dimensional spatial domain follows immediately. Here, $\bm{u}$ is the vector of conserved variables and $\bm{f},\bm{g}$ are the non-linear flux vectors. Examples of \eqref{eq:2DconsLaw} are numerous, including, e.g., the shallow water equations and the compressible Euler equations. The entropy of a non-linear hyperbolic system is an auxiliary conservation law for smooth solutions (and an inequality for discontinuous solutions), see \cite{Tadmor1987_2,Tadmor2003} for details. Given a strongly convex entropy function, $\Ent=\Ent(\bm u)$, there exists a set of entropy variables defined as
\begin{equation}\label{eq:entVars}
\bm{v} = \frac{\partial\Ent}{\partial \bm{u}}.
\end{equation}
Contracting the system of conservation laws \eqref{eq:2DconsLaw} from the left by the new set of variables \eqref{eq:entVars} yields a scalar conservation law for smooth solutions
\begin{equation}\label{eq:newSys}
\bm{v}^T\left(\bm{u}_t+ \bm{f}_x(\bm{u})+ \bm{g}_y(\bm{u})\right)= S_t + F_x + G_y= 0,
\end{equation}
provided certain compatibility conditions are satisfied between the physical fluxes $\bm{f}, \bm{g}$ and the entropy fluxes $F,G$ \cite{Tadmor1987_2,Tadmor2003}. In the presence of discontinuities the mathematical entropy decays \cite{Tadmor1987_2,Tadmor2003} and satisfies the inequality
\begin{equation}\label{eq:newIneq}
S_t + F_x + G_y \leq 0,
\end{equation}
in the sense of weak solutions to the non-linear PDE \cite{evans2010,Tadmor2003}. The final goal in this subsection is to determine a high-order DGSEM that is entropy stable on conforming meshes.

We first provide a brief overview for the derivation of the standard nodal DGSEM on rectangular grids. Complete details can be found in the book of Kopriva \cite{Kopriva:2009nx}. The DGSEM is derived from the weak form of the conservation laws \eqref{eq:2DconsLaw}. Thus, we multiply by an arbitrary ${L}_2(\Omega)$ test function $\varphi$ and integrate over the domain
\begin{equation}\label{eq:weakForm}
\int\limits_{\Omega}\left(\bm{u}_t+\bm{f}_x+\bm{g}_y\right)\varphi\dxdy = \bm{0},
\end{equation}
where, for convenience, we suppress the $\bm{u}$ dependence of the non-linear flux vectors. We subdivide the domain $\Omega$ into $K$ non-overlapping, geometrically conforming rectangular elements
\begin{equation}
E_k = \left[x_{k,1},x_{k,2}\right]\times[y_{k,1},y_{k,2}],\quad k = 1,\ldots,K.
\end{equation}
This divides the integral over the whole domain into the sum of the integrals over the elements. So, each element contributes
\begin{equation}\label{eq:weakForm2}
\int\limits_{E_k}\left(\bm{u}_t+\bm{f}_x+\bm{g}_y\right)\varphi\dxdy = \bm{0},\quad k = 1,\ldots K,
\end{equation}
to the total integral. Next, we create a scaling transformation between the reference element $E_0=[-1,1]^2$ and each element, $E_k$. For rectangular meshes we create mappings $(X_k,Y_k):E_0 \rightarrow E_k$ such that $\left(X_k(\xi),Y_k(\eta)\right) = (x,y)$ are defined as
\begin{equation}
 X_k(\xi) = x_{k,1} + \frac{\xi +1}{2}\Delta x_k,\quad  Y_k(\eta) = y_{k,1} + \frac{\eta +1}{2}\Delta y_k,
\label{eq:mapping}
\end{equation}
for $k=1,\ldots,K$ where $\Delta x_k=\left(x_{k,2}-x_{k,1}\right)$ and $\Delta y_k = \left(y_{k,2}-y_{k,1}\right)$. Under the transformation \eqref{eq:mapping} the conservation law in physical coordinates \eqref{eq:2DconsLaw} becomes a conservation law in reference coordinates \cite{Kopriva:2009nx}
\begin{equation}
\bm{u}_t + \frac{1}{J}\left[\tilde{\bm{f}}_{\xi} + \tilde{\bm{g}}_{\eta}\right] = \bm{0},
\end{equation}
where
\begin{equation}
J=\frac{\Delta x_k\Delta y_k}{4},\quad\tilde{\bm{f}} = \frac{\Delta y_k}{2}\bm{f},\quad\tilde{\bm{g}} = \frac{\Delta x_k}{2}\bm{g},
\end{equation}
and $k=1,\ldots,K$.

We select the test function $\varphi$ to be a piecewise polynomial of degree $N$ in each spatial direction
\begin{equation}\label{eq:testFunction}
\varphi^k = \sum_{i=0}^N\sum_{j=0}^N\varphi_{ij}^k\ell_i(\xi)\ell_j(\eta),
\end{equation}
on each spectral element $E_k$, but do not enforce continuity at the element boundaries. The interpolating Lagrange basis functions are defined by
\begin{equation}
\ell_i(\xi) = \prod_{\stackrel{j = 0}{j \neq i}}^N \frac{\xi-\xi_j}{\xi_i-\xi_j} \quad\text{for}\quad i=0,\ldots,N,
\label{eq:Lagrange}
\end{equation}
with a similar definition in the $\eta$ direction. The values of $\varphi^k_{ij}$ on each element $E_k$ are arbitrary and linearly independent, therefore the formulation \eqref{eq:weakForm2} is
\begin{equation}
\int\limits_{E_0}\left(J\bm{u}_t+\tilde{\bm{f}}_{\xi}+\tilde{\bm{g}}_{\eta}\right)\ell_i(\xi)\ell_j(\eta)\dxideta = \bm{0},
\end{equation}
where $i,j=0,\ldots,N$.

We approximate the conservative vector $\bm{u}$ and the contravariant fluxes $\tilde{\bm{f}}$, $\tilde{\bm{g}}$ with the same polynomial interpolants of degree $N$ in each spatial direction written in Lagrange form, e.g.,
\begin{equation}
\begin{aligned}
\bm{u}(x,y,t)|_{E_k} &= \bm{u}(\xi,\eta,t) \approx \sum_{i,j = 0}^N \bm{U}_{ij} \ell_i(\xi) \ell_j(\eta)\equiv \bm{U},
\\
\tilde{\bm{f}}\left(\bm{u}(x,y,t)\right)|_{E_k} &= \tilde{\bm{f}}(\xi,\eta,t) \approx  \sum_{i,j = 0}^N \tilde{\bm{F}}_{ij} \ell_i(\xi) \ell_j(\eta)\equiv \tilde{\bm{F}}.
\label{eq:DG-approx}
\end{aligned}
\end{equation}
Any integrals present in the DG approximation are approximated with a high-order Legendre-Gauss-Lobatto (LGL) quadrature rule, e.g.,
\begin{equation}\label{eq:quadrature}
\int\limits_{E_0} J\bm{U}_t\ell_i(\xi)\ell_j(\eta)\dxideta\approx J\!\!\!\sum_{n,m=0}^N\left(\sum_{p,q=0}^N\left(\bm{U}_t\right)_{pq}\ell_p(\xi_n)\ell_q(\eta_m)\right)\ell_i(\xi_n)\ell_j(\eta_m)\omega_n \omega_m
= J\left(\vec{U}_t\right)_{ij}\omega_i \omega_j,
\end{equation}
where $\left\{\xi_i\right\}_{i=0}^N, \left\{\eta_j\right\}_{j=0}^N$ are the LGL quadrature nodes and $\left\{\omega_i\right\}_{i=0}^N,\left\{\omega_j\right\}_{j=0}^N$ are the LGL quadrature weights. Further, we \textit{collocate} the interpolation and quadrature nodes which enables us to exploit that the Lagrange basis functions \eqref{eq:Lagrange} are discretely orthogonal and satisfy the Kronecker delta property, i.e., $\ell_j(\xi_i) = \delta_{ij}$ with $\delta_{ij} = 1$ for $i=j$ and $\delta_{ij}=0$ for $i\neq j$ to simplify \eqref{eq:quadrature}.

For spectral element methods where the nodes include the boundary of the reference space ($\xi_0=\eta_0=-1$ and $\xi_N=\eta_N=1$), the discrete derivative matrix $\D$ and the discrete mass matrix $\mass$ satisfy the summation-by-parts (SBP) property \cite{Carpenter1996}
\begin{equation}
\mass \D + (\mass \D)^T = \mat Q+\mat Q^T = \mat B := \text{diag} (-1,0,\ldots,0,1).
\label{eq:SBP}
\end{equation}
By considering LGL quadrature, we obtain a diagonal mass matrix
\begin{equation}
\mass=\text{diag} (\omega_0, \ldots , \omega_N),
\label{Mmatrix}
\end{equation}
with positive weights for any polynomial order \cite{Gassner2013}. Note, that the mass matrix is constructed by performing mass lumping. We also define the SBP matrix $\mat Q$ and the boundary matrix $\mat B$ in \eqref{eq:SBP}. The SBP property \eqref{eq:SBP} gives the relation
\begin{equation}\label{eq:otherSBP}
\D = \mass^{-1}\mat B - \mass^{-1}\D^T\mass,
\end{equation}
where we use the fact that the mass matrix $\mass$ is positive definite and invertible. 

By rewriting the polynomial derivative matrix as \eqref{eq:otherSBP} we can move discrete derivatives off the contravariant fluxes and onto the test function. This generates surface and volume contributions in the approximation. To resolve the discontinuities that naturally occur at element interfaces in DG methods we introduce the numerical flux functions $\tilde{\bm{F}}^*,\tilde{\bm{G}}^*$. We apply the SBP property \eqref{eq:otherSBP} again to move derivatives off the test function back onto the contravariant fluxes. This produces the strong form of the nodal DGSEM
\begin{equation}\label{eq:standardDG}
\begin{aligned}
J\left(\bm{U}_t\right)_{ij} &+ \frac{1}{\mass_{ii}}\left(\delta_{iN}\left[\tilde{\bm{F}}^{*}(1,\eta_j;\hat{n}) - \tilde{\bm{F}}_{Nj}\right] - \delta_{i0}\left[\tilde{\bm{F}}^{*}(-1,\eta_j;\hat{n}) - \tilde{\bm{F}}_{0j}\right]\right)+\sum_{m=0}^N \D_{im}\tilde{\bm{F}}_{mj}\\
&+ \frac{1}{\mass_{jj}}\left(\delta_{jN}\left[\tilde{\bm{G}}^{*}(\xi_i,1;\hat{n}) - \tilde{\bm{G}}_{iN}\right] - \delta_{j0}\left[\tilde{\bm{G}}^{*}(\xi_i,-1;\hat{n}) - \tilde{\bm{G}}_{i0}\right]\right)+\sum_{m=0}^N \D_{jm}\tilde{\bm{G}}_{im}=\bm{0},\\
\end{aligned}
\end{equation}
for each LGL node with $i,j=0,\ldots,N$. We introduce notation in \eqref{eq:standardDG} for the evaluation of the contravariant numerical flux functions in the normal direction along each edge of the reference element at the relevant LGL nodes, e.g. $\tilde{\bm{F}}^{*}(1,\eta_j;\hat{n})$ for $j=0,\ldots,N$.  Note that selecting the test function to be the tensor product basis \eqref{eq:testFunction} decouples the derivatives in each spatial direction.

Next, we extend the standard strong form DGSEM \eqref{eq:standardDG} into a split form DGSEM \cite{Carpenter2014,Gassner2016} framework. Split formulations of the DG approximation offer increased robustness, e.g. \cite{Gassner2016,Yee2017}, as well as increased flexibility in the DGSEM to satisfy auxiliary properties such as entropy conservation or entropy stability \cite{Carpenter2014,Gassner2016,Ray2017}. To create a split form DGSEM we rewrite the contributions of the volume integral, for example in the $\xi-$direction, by
\begin{equation}\label{eq:newVolInt}
\sum_{m=0}^N \D_{im}\tilde{\bm{F}}_{mj} \approx 2\sum_{m=0}^N \D_{im}\tilde{\bm{F}}^{\#}\left(\bm{U}_{ij},\bm{U}_{mj}\right),
\end{equation}
for $i,j=0,\ldots,N$ where we introduce a two-point, symmetric numerical volume flux $\tilde{\bm{F}}^{\#}$ \cite{Gassner2016}. This step creates a baseline split form DGSEM
\begin{equation}\label{eq:splitDG}
\resizebox{\hsize}{!}{$
\begin{aligned}
J\left(\bm{U}_t\right)_{ij} &+ \frac{1}{\mass_{ii}}\left(\delta_{iN}\left[\tilde{\bm{F}}^{*}(1,\eta_j;\hat{n}) - \tilde{\bm{F}}_{Nj}\right] - \delta_{i0}\left[\tilde{\bm{F}}^{*}(-1,\eta_j;\hat{n}) - \tilde{\bm{F}}_{0j}\right]\right)+2\sum_{m=0}^N \D_{im}\tilde{\bm{F}}^{\#}\left(\bm{U}_{ij},\bm{U}_{mj}\right)\\
&+ \frac{1}{\mass_{jj}}\left(\delta_{jN}\left[\tilde{\bm{G}}^{*}(\xi_i,1;\hat{n}) - \tilde{\bm{G}}_{iN}\right] - \delta_{j0}\left[\tilde{\bm{G}}^{*}(\xi_i,-1;\hat{n}) - \tilde{\bm{G}}_{i0}\right]\right)+2\sum_{m=0}^N \D_{jm}\tilde{\bm{G}}^{\#}\left(\bm{U}_{ij},\bm{U}_{im}\right)=\bm{0},\\
\end{aligned}$}
\end{equation}
that can be used to create an entropy conservative/stable approximation. All that remains is the precise definition of the numerical surface and volume flux functions.

The construction of a high-order entropy conserving/stable DGSEM relies on the fundamental finite volume framework developed by Tadmor \cite{tadmor:1984,Tadmor1987_2}. An entropy conservative (EC) numerical flux function in the $\xi-$direction, $f^{\EC}$, is derived by satisfying the condition \cite{Tadmor2003}
\begin{equation}\label{eq:entCondition}
\jump{\bm{v}}^T\bm{f}^{\EC} = \jump{\Psi^f},
\end{equation}
where $\bm{v}$ are the entropy variables \eqref{eq:entVars}, $\Psi^{f}$ is the entropy flux potential
\begin{equation}\label{eq:entPotential}
\Psi^f = \bm{v}\cdot\bm{f} - F,
\end{equation}
and
\begin{equation}\label{jump}
\jump{\cdot} = (\cdot)_R - (\cdot)_L,
\end{equation}
is the jump operator between a left and right state. Note that \eqref{eq:entCondition} is a single condition on the numerical flux vector $\bm{f}^{\EC}$, so there are many potential solutions for the entropy conserving flux vector. However, we reduce the number of possible solutions with the additional requirement that the numerical flux must be consistent, i.e. $\bm{f}^{\EC}(\bm{u},\bm{u}) = \bm{f}(\bm{u})$. Many such entropy conservative numerical flux functions are available for systems of hyperbolic conservation laws, e.g. the Euler equations \cite{Chandra2013,Ismail2009}. The entropy conservative flux function creates a baseline scheme to which dissipation can be added and guarantee discrete satisfaction of the entropy inequality (entropy stability), e.g. \cite{Chandra2013,Fjordholm2011,Wintermeyer2016}.

Remarkably, Fisher et al. \cite{Fisher2013} and Fisher and Carpenter \cite{Fisher2013b} demonstrated that selecting an entropy conservative finite volume flux for the numerical surface and volume fluxes in a high-order SBP discretization is enough to guarantee that the property of entropy conservation remains. As mentioned earlier, the DGSEM constructed on the LGL nodes is an SBP method. Entropy stability of the high-order DG approximation is guaranteed by adding proper numerical dissipation in the numerical surface fluxes, similar to the finite volume case. Thus, the final form of the entropy conservative DGSEM on conforming meshes is
\begin{equation}\label{eq:ECDG}
\resizebox{\hsize}{!}{$
\begin{aligned}
J\left(\bm{U}_t\right)_{ij} &+ \frac{1}{\mass_{ii}}\left(\delta_{iN}\left[\tilde{\bm{F}}^{\EC}(1,\eta_j;\hat{n}) - \tilde{\bm{F}}_{Nj}\right] - \delta_{i0}\left[\tilde{\bm{F}}^{\EC}(-1,\eta_j;\hat{n}) - \tilde{\bm{F}}_{0j}\right]\right)+2\sum_{m=0}^N \D_{im}\tilde{\bm{F}}^{\EC}\left(\bm{U}_{ij},\bm{U}_{mj}\right)\\
&+ \frac{1}{\mass_{jj}}\left(\delta_{jN}\left[\tilde{\bm{G}}^{\EC}(\xi_i,1;\hat{n}) - \tilde{\bm{G}}_{iN}\right] - \delta_{j0}\left[\tilde{\bm{G}}^{\EC}(\xi_i,-1;\hat{n}) - \tilde{\bm{G}}_{i0}\right]\right)+2\sum_{m=0}^N \D_{jm}\tilde{\bm{G}}^{\EC}\left(\bm{U}_{ij},\bm{U}_{im}\right)=\bm{0},\\
\end{aligned}$}
\end{equation}
where we have made the replacement of the numerical surface and volume fluxes to be a two-point, symmetric EC flux that satisfies \eqref{eq:entCondition}. \\

\begin{rem}
We note that the entropy conservative DGSEM \eqref{eq:ECDG} is equivalent to a SBP finite difference method with boundary coupling through simultaneous approximation terms (SATs), e.g. \cite{Fisher2013b,Fisher2013}.
\end{rem}

In summary, we demonstrated that special attention was required for the volume contribution in the nodal DGSEM to create a split form entropy conservative method. Additionally, the SBP property was necessary to apply previous results from Fisher et al. \cite{Fisher2013} and guarantee entropy conservation at high-order. For the conforming mesh case the surface contributions required little attention. We simply replaced the numerical surface flux with an appropriate EC flux from the finite volume community. However, we next consider non-conforming DG methods with the flexibility to have differing polynomial order or hanging nodes at element interfaces.

\subsection{Non-Conforming DGSEM}

We consider the standard DGSEM in strong form \eqref{eq:standardDG} to discuss the commonly used \textit{mortar method} for non-conforming high-order DG approximations \cite{Tan2012,Kopriva1996b}. The mortar method allows for the polynomial order to differ between elements (Fig. \ref{fig:NonConfMesh}(a)), sometimes called $p$ refinement or algebraic non-conforming, as well as meshes that contain hanging nodes (Fig. \ref{fig:NonConfMesh}(b)), sometimes called $h$ refinement or geometric non-conforming, or both for a fully $h/p$ non-conforming approach (Fig. \ref{fig:NonConfMesh}(c)). For ease of presentation we assume that the polynomial order within an element is the same in each spatial direction. Note, however, due to the tensor product decoupling of the approximation (e.g. \eqref{eq:ECDG}) the mortar method could allow the polynomial order to differ within an element in each direction $\xi$ and $\eta$ as well.
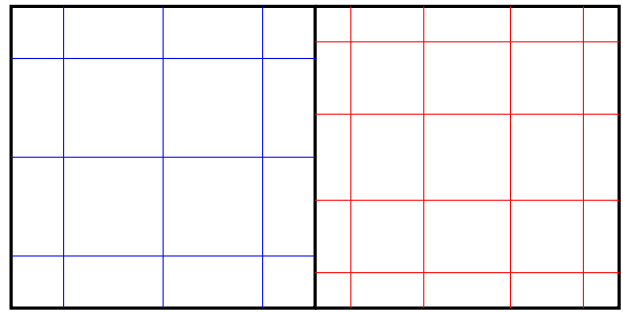
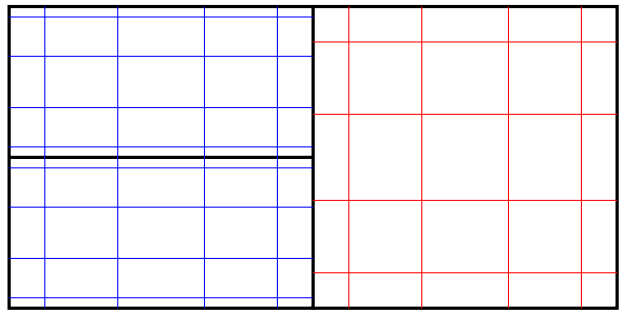
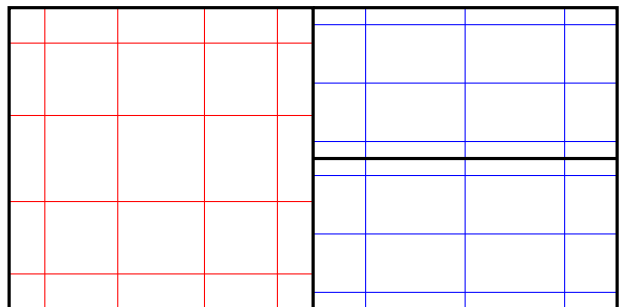
\begin{figure}[!ht]
\begin{center}
\subfloat[Neighboring elements with different nodal distributions]{
\begin{tikzpicture}[scale=1]
\draw[very thick] (-4,-2) -- (0,-2) -- (0,2) -- (-4,2) -- cycle;
\draw[very thick] (0,-2) -- (4,-2) -- (4,2) -- (0,2) ;
\color[rgb]{0,0,1}
\draw[] (-3.3093,2) -- (-3.3093,-2) ;
\draw[] (-2,2) -- (-2,-2) ;
\draw[] (-0.6907,2) -- (-0.6907,-2) ;
\draw[] (-4,1.3093) -- (0,1.3093) ;
\draw[] (-4,0) -- (0,0) ;
\draw[] (-4,-1.3093) -- (0,-1.3093) ;
\color[rgb]{1,0,0}
\draw[] (3.5301,2) -- (3.5301,-2) ;
\draw[] (2.5705,2) -- (2.5705,-2) ;
\draw[] (1.4295,2) -- (1.4295,-2) ;
\draw[] (0.4699,2) -- (0.4699,-2) ;
\draw[] (4,1.5301) -- (0,1.5301) ;
\draw[] (4,0.5705) -- (0,0.5705) ;
\draw[] (4,-1.5301) -- (-0,-1.5301) ;
\draw[] (4,-0.5705) -- (-0,-0.5705) ;
\end{tikzpicture}
}
\\
\subfloat[Neighboring elements with a hanging node]{
\begin{tikzpicture}[scale=1]
\draw[very thick] (-4,-2) -- (0,-2) -- (0,2) -- (-4,2) -- cycle;
\draw[very thick] (0,-2) -- (4,-2) -- (4,2) -- (0,2) ;
\draw[very thick] (-4,0) -- (0,0);
\color[rgb]{0,0,1}
\draw[] (-3.5301,2) -- (-3.5301,-2) ;
\draw[] (-2.5705,2) -- (-2.5705,-2) ;
\draw[] (-1.4295,2) -- (-1.4295,-2) ;
\draw[] (-0.4699,2) -- (-0.4699,-2) ;
\draw[] (-4,-1.8611) -- (0,-1.8611);
\draw[] (-4,-1.3399) -- (0,-1.3399);
\draw[] (-4,-0.1389) -- (0,-0.1389);
\draw[] (-4,-0.6601) -- (0,-0.6601);
\draw[] (-4,1.3399) -- (0,1.3399);
\draw[] (-4,1.8611) -- (0,1.8611);
\draw[] (-4,0.1389) -- (0,0.1389);
\draw[] (-4,0.6601) -- (0,0.6601);
\color[rgb]{1,0,0}
\draw[] (3.5301,2) -- (3.5301,-2) ;
\draw[] (2.5705,2) -- (2.5705,-2) ;
\draw[] (1.4295,2) -- (1.4295,-2) ;
\draw[] (0.4699,2) -- (0.4699,-2) ;
\draw[] (4,1.5301) -- (0,1.5301) ;
\draw[] (4,0.5705) -- (0,0.5705) ;
\draw[] (4,-1.5301) -- (-0,-1.5301) ;
\draw[] (4,-0.5705) -- (-0,-0.5705) ;
\end{tikzpicture}
}
\\
\subfloat[Neighboring elements with a hanging node and differing polynomial order]{
\begin{tikzpicture}[scale=1]
\color[rgb]{0,0,1}
\draw[] (7.3093,2) -- (7.3093,-2) ;
\draw[] (6,2) -- (6,-2) ;
\draw[] (4.6907,2) -- (4.6907,-2) ;
\draw[] (4,1) -- (8,1);
\draw[] (4,1.774597) -- (8,1.774597);
\draw[] (4,0.2254) -- (8,0.2254);
\draw[] (4,-1) -- (8,-1);
\draw[] (4,-1.774597) -- (8,-1.774597);
\draw[] (4,-0.2254) -- (8,-0.2254);
\color[rgb]{1,0,0}
\draw[] (3.5301,2) -- (3.5301,-2) ;
\draw[] (2.5705,2) -- (2.5705,-2) ;
\draw[] (1.4295,2) -- (1.4295,-2) ;
\draw[] (0.4699,2) -- (0.4699,-2) ;
\draw[] (4,1.5301) -- (0,1.5301) ;
\draw[] (4,0.5705) -- (0,0.5705) ;
\draw[] (4,-1.5301) -- (-0,-1.5301) ;
\draw[] (4,-0.5705) -- (-0,-0.5705) ;
\color[rgb]{0,0,0}
\draw[very thick] (0,-2) -- (4,-2) -- (4,2) -- (0,2) --cycle;
\draw[very thick] (4,-2) -- (8,-2) -- (8,2) -- (4,2);
\draw[very thick] (4,0) -- (8,0);
\end{tikzpicture}
}
\caption{Examples of simple meshes with (a) $p$ refinement (b) $h$ refinement or (c) $h/p$ refinement}
\label{fig:NonConfMesh}
\end{center}
\end{figure}

The key to the non-conforming spectral element approximation is how the numerical fluxes between neighbor interfaces are treated. In the conforming approximation of the previous section the interface points between two neighboring elements coincide while the numerical solution across the interface was discontinuous. This allowed for a straightforward definition of unique numerical surface fluxes to account for how information is transferred between neighbors. It is then possible to determine numerical surface fluxes that guaranteed entropy conservation/stability of the conforming approximation.

The only difference between the conforming and non-conforming approximations is precisely how the numerical surface fluxes are computed along the interfaces. In the non-conforming cases of $h/p$ refinement (Fig. \ref{fig:NonConfMesh}(a)-(c)), the interface nodes may not match. So, a point-by-point transfer of information cannot be made between an element and its neighbors. To remedy this the mortar method ``cements'' together the neighboring ``bricks'' by connecting them through an intermediate one-dimensional construct, denoted by $\Xi$, see Fig. \ref{fig:MortarIdea}(a)-(b).
\begin{figure}[!ht]
\begin{center}
\subfloat[Three element $h/p$ non-conforming mesh]{
\begin{tikzpicture}[scale=1.0]
\draw[very thick] (-1,0) rectangle (5.25,4.25);
\draw[very thick] (2.5,0) -- (2.5,4.25);
\draw[very thick] (-1.0,2.15) -- (2.5,2.15);
\draw (0.825,3.15) node {$E_1$};
\draw (0.825,1.1) node {$E_2$};
\draw (3.8,2.125) node {$E_3$};
\end{tikzpicture}
}
\quad
\subfloat[Mortar projections]{
\begin{tikzpicture}[scale=1.0]
\draw[very thick] (0,3)  rectangle (4,5);
\draw[very thick] (0,-1) rectangle (4,1);
\draw[very thick] (7,-1) rectangle (10,5);
\draw[very thick,pattern= north east lines] (-0.15,1.85) rectangle (4.15,2.15);
\draw[very thick,pattern= north east lines] (5.35,2.85)  rectangle (5.65,5.15);
\draw[very thick,pattern= north east lines] (5.35,-1.15)  rectangle (5.65,1.15);
\draw (2,4)        node {$E_1$};
\draw (2,0)        node {$E_2$};
\draw (8.5,2)      node {$E_3$};
\draw (-0.5,2)     node {$\Xi_{12}$};
\draw (5.55,5.45)  node {$\Xi_{13}$};
\draw (5.55,-1.45) node {$\Xi_{23}$};
\draw[very thick,triangle 45-] (5.3,4.5)  -- (4.1,4.5);
\draw[very thick,triangle 45-] (4.05,3.5) -- (5.25,3.5);
\draw[very thick,triangle 45-] (5.7,4.5) -- (6.9,4.5);
\draw[very thick,triangle 45-] (6.925,3.5) -- (5.75,3.5);
\draw[very thick,triangle 45-] (5.3,0.5) -- (4.1,0.5);
\draw[very thick,triangle 45-] (4.05,-0.5) -- (5.25,-0.5);
\draw[very thick,triangle 45-] (5.7,0.5) -- (6.9,0.5);
\draw[very thick,triangle 45-] (6.925,-0.5) -- (5.75,-0.5);
\draw[very thick,triangle 45-] (1,1.8) -- (1,1.1);
\draw[very thick,triangle 45-] (3,1.1) -- (3,1.8);
\draw[very thick,triangle 45-] (1,2.2) -- (1,2.9);
\draw[very thick,triangle 45-] (3,2.9) -- (3,2.2);
\end{tikzpicture}
}
\caption{Diagram depicting communication of data to and from mortars between three non-conforming elements.}
\label{fig:MortarIdea}
\end{center}
\end{figure}
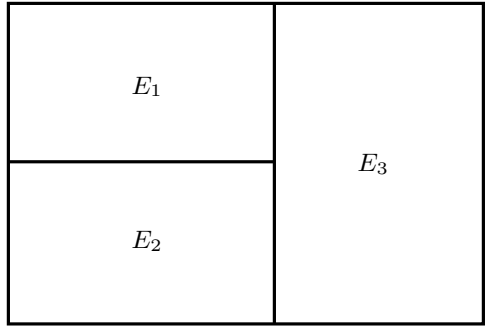
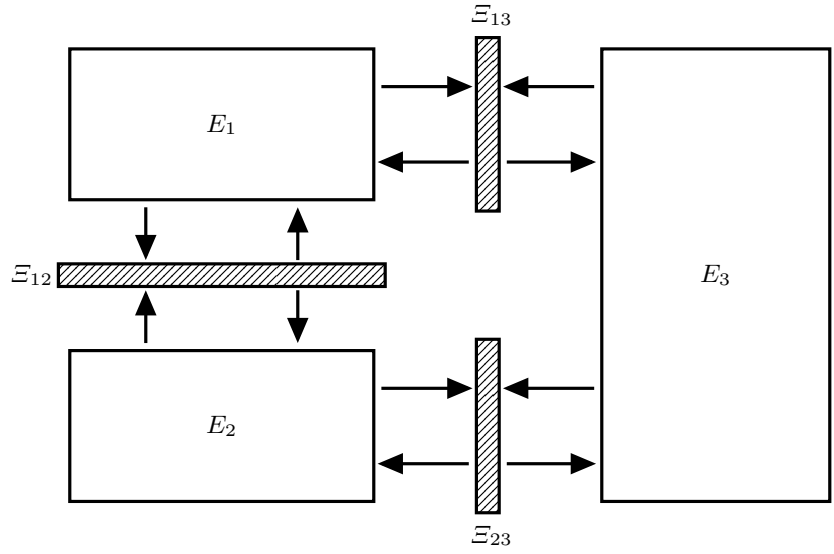

In this overview we only discuss the coupling of the $p$ refinement case (Fig. \ref{fig:NonConfMesh}(a)), but the process is similar for the $h$ refinement case and is nicely described by Kopriva \cite{Kopriva1996b}. Also, the extension to curvilinear elements is briefly outlined. We distinguish them as the polynomial order on the left, $N_L$, and right, $N_R$. The polynomial on the mortar is chosen to be $N_{\Xi} = \max(N_L,N_R)$ \cite{Kopriva1996b,Kopriva2002}. Without loss of generality we assume that $N_L<N_R$, as depicted in Fig. \ref{fig:NonConfMesh}(a), such that $N_{\Xi}=N_R$. The construction of the numerical flux at such a non-conforming interface follows three basic steps:
\begin{enumerate}
\item Because the polynomial order on the right ($R$) and the mortar match, we simply copy the data. From the left ($L$) element we use a discrete or exact $L_2$ projection to move the solution from the element onto the mortar $\Xi$.
\item The node distributions on the mortar match and we compute the interface numerical flux similar to the conforming mesh case.
\item Finally, we project the numerical flux from the mortar back to each of the elements. Again, the left element uses a discrete or exact $L_2$ projection and the right element simply copies the data.
\end{enumerate}
We collect these steps visually in Fig. \ref{fig:projections} and introduce the notation for the four projection operations to be $\PLtoX$, $\PXtoL$, $\PRtoX$, $\PXtoR$. For this example we note that the right to mortar and inverse projections are the appropriate sized identity matrix, i.e $\PRtoX = \PXtoR = \mat{I}_{N_R}$. We provide additional details in Appendix \ref{sec:App B} regarding the mortar method for $p$ non-conforming DG methods and clarify the difference between interpolation and projection operators.
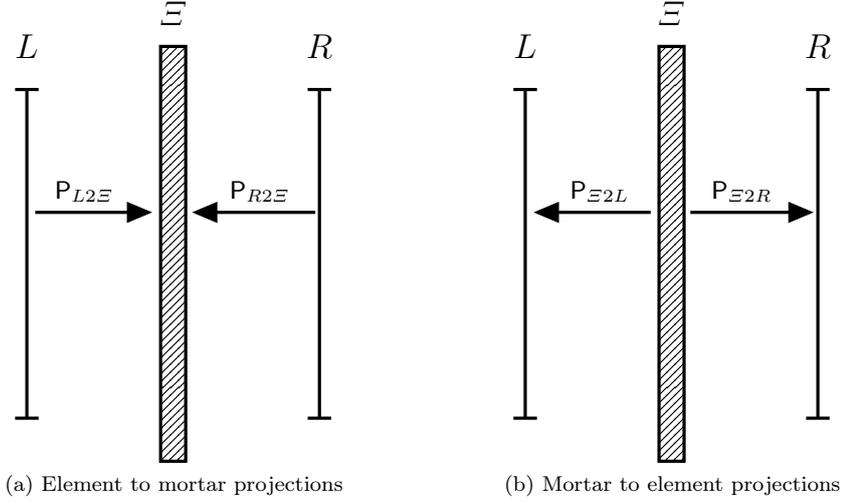
\begin{figure}[!ht]
\begin{center}
\subfloat[Element to mortar projections]{
\begin{tikzpicture}[scale=1.1]
\draw[very thick,|-|] (-1,2) -- (-1,-2);
\draw[very thick,|-|] (2.5,-2) -- (2.5,2);
\draw[very thick,pattern= north east lines] (0.9,-2.5) rectangle (0.6,2.5);
\draw (2.5,2.5) node {\LARGE{$R$}};
\draw (-1,2.5)  node {\LARGE{$L$}};
\draw (0.75,2.9) node {\LARGE{$\Xi$}};
\draw[very thick,triangle 45-] (0.975,0.5) -- (2.45,0.5);
\draw[very thick,triangle 45-] (0.5,0.5) -- (-0.9,0.5);
\draw (-0.3125,0.75) node {$\PLtoX$};
\draw ( 1.7875,0.75) node {$\PRtoX$};
\end{tikzpicture}
}\qquad\qquad\qquad
\subfloat[Mortar to element projections]{
\begin{tikzpicture}[scale=1.1]
\draw[very thick,|-|] (-1,2) -- (-1,-2);
\draw[very thick,|-|] (2.5,-2) -- (2.5,2);
\draw[very thick,pattern= north east lines] (0.9,-2.5) rectangle (0.6,2.5);
\draw (2.5,2.5) node {\LARGE{$R$}};
\draw (-1,2.5)  node {\LARGE{$L$}};
\draw (0.75,2.9) node {\LARGE{$\Xi$}};
\draw[very thick,triangle 45-] (2.45,0.5) -- (0.975,0.5);
\draw[very thick,triangle 45-] (-0.9,0.5) -- (0.5,0.5);
\draw (-0.1125,0.75) node {$\PXtoL$};
\draw ( 1.5875,0.75) node {$\PXtoR$};
\end{tikzpicture}
}
\caption{Schematic of mortar projections for the case of $p$ refinement.}
\label{fig:projections}
\end{center}
\end{figure}

\subsection{Interaction of the Standard Mortar Method with Entropy Conservative DGSEM}

With the machinery of the mortar method now in place to handle non-conforming interfaces we are equipped to revisit the discussion of the entropy conservative DGSEM. For linear problems, where entropy conservation becomes \textit{energy} conservation, it is known that the mortar method is sufficient to extend the energy conserving DG schemes to non-conforming meshes, e.g. \cite{Friedrich2016,Kozdon2016,Mattsson2010b}. This is because no non-linearities are present and there is no coupling of the left and right solution states in the central numerical flux. However, for non-linear problems we replace this simple central numerical flux with a more complicated entropy conservative numerical flux that features possible polynomial or rational non-linearities as well as strong cross coupling between the left and right solution states, e.g. \cite{Chandra2013,Fjordholm2011,Gassner2013}. This introduces complications when applying the standard mortar method to entropy conservative DG methods.

As a simple example, consider the Burgers' equation which is equipped with an entropy conservative numerical flux in the $\xi-$direction of the form \cite{Gassner2013}
\begin{equation}
F^{\EC} = \frac{1}{6}\left(U_L^2+U_LU_R+U_R^2\right).
\end{equation}
Continuing the assumption of $N_L<N_R$ from the previous subsection we find the numerical flux computed on the mortar is
\begin{equation}\label{eq:BurgersFlux}
F^{\EC}_{\Xi} = \frac{1}{6}\left[\left(\PLtoX U_L\right)^2+\left(\PLtoX U_L\right)U_R+U_R^2\right].
\end{equation}
The back projections of the mortar numerical flux \eqref{eq:BurgersFlux} onto the left and right elements are
\begin{equation}
F^{\EC}_L = \PXtoL F^{\EC}_{\Xi},\quad F^{\EC}_{R}=F^{\EC}_{\Xi}.
\end{equation}
However, it is clear that the projected numerical fluxes will exhibit unpredictable behavior with regards to entropy. For example, because the entropy conservative flux was derived for conforming meshes with point-to-point information transfer, it is not obvious how the operation to compute the square of the projection of $U_L$ and then $L_2$ project the numerical flux back to the left element will change the entropy.

The focus of this article is to remedy these issues and happily marry the entropy conservative DGSEM with a $h/p$ non-conforming mortar-type method. To achieve this goal requires careful consideration and construction of the projection operators to move solution information between non-conforming element neighbors. Our main results are presented in the next section. First, in Sec. \ref{sec:p}, we address the issues associated with $p$ refinement similar to Carpenter et al. \cite{Carpenter2016} only in the context of a split form DG framework. We build on the $p$ refinement result to construct projections that guarantee entropy conservation for in the case of $h$ refinement in Sec. \ref{sec:h}. Then, Sec. \ref{sec:Dissipation} describes how additional dissipation can be included at non-conforming interfaces to guarantee entropy stability. Finally, we verify the theoretical derivations through a variety of numerical test cases in Sec. \ref{sec:numResults}.

\section{Entropy Stable $h/p$ Non-Conforming DGSEM}\label{sec:interfaces}

Our goal is to develop a high-order numerical approximation that conserves the primary quantities of interest (like mass) as well as obey the second law of thermodynamics. In the continuous analysis, neglecting boundary conditions, we know for general solutions that the main quantities are conserved and the entropy can be dissipated (in the mathematical sense)
\begin{equation}\label{BothConditionsCon}
\int\limits_\Omega {u}^{\eqInd}_t~\textrm{d}\Omega = 0,\qquad
\int\limits_\Omega \dEnt~\textrm{d}\Omega \le 0,
\end{equation}
for each equation, $\eqInd=1,\ldots,M$, in the non-linear system. We aim to develop a DGSEM that mimics \eqref{BothConditionsCon} on rectangular meshes in the case of general $h/p$ non-conforming approximations.

As discussed previously, the most important component of a non-conforming method for entropy stable approximations is the coupling of the solution at interfaces through numerical fluxes. For convenience we clarify the notation of the numerical fluxes in the entropy conservative approximation \eqref{eq:ECDG} along interfaces in Fig. \ref{fig:Numflux}.
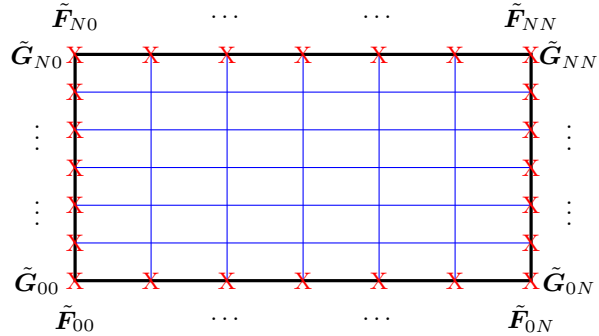
\begin{figure}[!ht]
\begin{center}
\begin{tikzpicture}[scale=1]
\draw[very thick] (-3,0) -- (3,0) -- (3,3) -- (-3,3) -- cycle;
\color[rgb]{0,0,1}
\draw[] (-2,3) -- (-2,0) ;
\draw[] (-1,3) -- (-1,0) ;
\draw[] (0,3) -- (0,0) ;
\draw[] (1,3) -- (1,0) ;
\draw[] (2,3) -- (2,0) ;
\draw[] (-3,2.5) -- (3,2.5) ;
\draw[] (-3,2) -- (3,2) ;
\draw[] (-3,1.5) -- (3,1.5) ;
\draw[] (-3,1) -- (3,1) ;
\draw[] (-3,0.5) -- (3,0.5) ;
\color[rgb]{1,0,0}
\node at (-3,0) {X} ;
\node at (-3,0.5) {X} ;
\node at (-3,1) {X} ;
\node at (-3,1.5) {X} ;
\node at (-3,2) {X} ;
\node at (-3,2.5) {X} ;
\node at (-3,3) {X} ;
\node at (3,0) {X} ;
\node at (-2,0) {X} ;
\node at (-1,0) {X} ;
\node at (0,0) {X} ;
\node at (1,0) {X} ;
\node at (2,0) {X} ;
\node at (3,2.5) {X} ;
\node at (3,2) {X} ;
\node at (3,1.5) {X} ;
\node at (3,1) {X} ;
\node at (3,0.5) {X} ;
\node at (3,3) {X} ;
\node at (2,3) {X} ;
\node at (1,3) {X} ;
\node at (0,3) {X} ;
\node at (-1,3) {X} ;
\node at (-2,3) {X} ;
\color[rgb]{0,0,0}
\node at (-3,-0.5) {$\tilde{\bm{F}}_{00}$} ;
\node at (-1,-0.5) {\dots} ;
\node at (1,-0.5) {\dots} ;
\node at (3,-0.5) {$\tilde{\bm{F}}_{0N}$} ;
\node at (-3,3.5) {$\tilde{\bm{F}}_{N0}$} ;
\node at (-1,3.5) {\dots} ;
\node at (1,3.5) {\dots} ;
\node at (3,3.5) {$\tilde{\bm{F}}_{NN}$} ;
\node at (-3.5,0) {$\tilde{\bm{G}}_{00}$} ;
\node at (-3.5,1) {\vdots} ;
\node at (-3.5,2) {\vdots} ;
\node at (-3.5,3) {$\tilde{\bm{G}}_{N0}$} ;
\node at (3.5,0) {$\tilde{\bm{G}}_{0N}$} ;
\node at (3.5,1) {\vdots} ;
\node at (3.5,2) {\vdots} ;
\node at (3.5,3) {$\tilde{\bm{G}}_{NN}$} ;
\end{tikzpicture}
\caption{Entropy conservative numerical fluxes at the interfaces of an element.}
\label{fig:Numflux}
\end{center}
\end{figure}

We seek an approximation that discretely preserves primary conservation and discrete entropy stability. The definition of this continuous property \eqref{BothConditionsCon} is translated into the discrete by summing over all elements to be
\begin{align}
\sum\limits_{\mathrm{all\,elements}} J\! \sum_{i,j=0}^N\omega_i\omega_j \left({U}^\eqInd_{t}\right)_{ij}& = {0}, \text{ (primary conservation)}\label{BothConditionsMom},\\
\sum\limits_{\mathrm{all\,elements}} J\! \sum_{i,j=0}^N\omega_i\omega_j \left(S_{t}\right)_{ij}& \le 0, \text{ (entropy stability)}\label{BothConditionsEntStab},
\end{align}
where $\left(S_{t}\right)_{ij}$ is a discrete evaluation of the time derivative of the entropy function.

While our goal is the construction of an entropy stable scheme, we will first derive an entropy conservative scheme for smooth solutions, meaning that
\begin{align}
\sum\limits_{\mathrm{all\,elements}} J\! \sum_{i,j=0}^N\omega_i\omega_j \left(S_{t}\right)_{ij}=0, \text{ (entropy conservation)}.\label{BothConditionsEntCon}
\end{align}
After deriving an entropy conservative scheme we can obtain an entropy stable scheme by including carefully constructed dissipation within the numerical surface flux as described in Sec. \ref{sec:Dissipation}.

To derive an approximation which conserves the primary quantities and is entropy stable we must examine the discrete growth in the primary quantities and entropy in a single element. \\

\begin{lem}\label{thm1} We assume that the two point volume flux satisfies the entropy conservation condition \eqref{eq:entCondition}. The discrete growth on a single element of the primary quantities and the entropy of the DG discretization \eqref{eq:ECDG} are
\begin{align}
J \sum_{i,j=0}^N\omega_i\omega_j\left({U}^\eqInd_{t}\right)_{ij}=&-\sum_{j=0}^N\omega_j\left(\tilde{{F}}^{\EC,\eqInd}_{Nj}-\tilde{{F}}^{\EC,\eqInd}_{0j}\right)-\sum_{i=0}^N\omega_i\left(\tilde{{G}}^{\EC,\eqInd}_{iN}-\tilde{{G}}^{\EC,\eqInd}_{i0}\right),\label{ThmdU}
\end{align}
where $\eqInd=1,\ldots,M$ and
\begin{align}
J \sum_{i,j=0}^N\omega_i\omega_j\left(\Ent_{t}\right)_{ij}=&-\sum_{j=0}^N\omega_j\left(\sum\limits_{\eqInd=1}^M{V}_{Nj}^{\eqInd}\tilde{{F}}^{\EC,\eqInd}_{Nj}-\tilde{\Psi}^f_{Nj}-\left(\sum\limits_{\eqInd=1}^M{V}_{0j}^{\eqInd}\tilde{{F}}^{\EC,\eqInd}_{0j}-\tilde{\Psi}^f_{0j}\right)\right)\notag\\
&-\sum_{i=0}^N\omega_i\left(\sum\limits_{\eqInd=1}^M{V}_{iN}^{\eqInd}\tilde{{G}}^{\EC,\eqInd}_{iN}-\tilde{\Psi}^g_{iN}-\left(\sum\limits_{\eqInd=1}^M{V}_{i0}^{\eqInd}\tilde{{G}}^{\EC,\eqInd}_{i0}-\tilde{\Psi}^g_{i0}\right)\right)\label{ThmdEnt},
\end{align}
rescpectively.
\end{lem}
\begin{proof} The proof of \eqref{ThmdU} and \eqref{ThmdEnt} is given in Fisher et al. \cite{Fisher2013}, however, for completeness, we included the proof consistent to the current notation and formulations in Appendix \ref{sec:App A}.
\end{proof}

We first examine the volume contributions of the entropy conservative approximation because when contracted into entropy space the volume terms move to the interfaces \cite{Fisher2013,Gassner2017} in the form of the entropy flux potential, i.e. \eqref{eq:entPotential}. Note, that the proof in Appendix \ref{sec:App A} concerns the contribution of the volume integral in the DGSEM and only depends on the interior of an element. Therefore, the result of Lemma \ref{thm1} holds for conforming as well as non-conforming meshes.

Therefore, to obtain a primary and entropy conservative scheme on the entire domain we need to choose an appropriate numerical surface flux. In comparison to the volume flux, the surface flux depends on the interfaces of the elements. Here, we need to differ between elements with conforming and non-conforming interfaces. We will first describe how to determine such a scheme for conforming interfaces, but differing polynomial orders. Then, we extend these results to consider meshes with non-conforming interfaces (hanging nodes).

\subsection{Conforming Interfaces}\label{sec:p}

In this section we will show how to create a fully conservative scheme on a standard conforming mesh, i.e. the polynomial orders match and there are no hanging nodes. As shown in \eqref{ThmdU} and \eqref{ThmdEnt} the primary conservation and entropy growth is only determined by the numerical surface fluxes on the interface. Here we exploit that the tensor product basis decouples the approximation in the two spatial directions and many of the proofs only address the $\xi-$direction because the contribution in the $\eta-$direction is done in an analogous fashion. Furthermore, the contribution at the four interfaces of an element follow similar steps. As such, we elect to consider all terms related to a single shared interface of a left and right element.
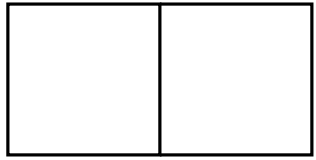
\begin{figure}[!ht]
\begin{center}
\begin{tikzpicture}[scale=0.5]
\draw[very thick] (-4,-2) -- (0,-2) -- (0,2) -- (-4,2) -- cycle;
\draw[very thick] (0,-2) -- (4,-2) -- (4,2) -- (0,2) ;
\end{tikzpicture}
\caption{Two neighboring elements with a single coinciding interface.}
\label{Confmesh}
\end{center}
\end{figure}
For a simple example we present a two element mesh in Fig. \ref{Confmesh} and consider the coupling through the single shared interface. Due to Lemma \ref{thm1} the terms referring to the shared interface are
\begin{align}
\ICon&=\sum_{j=0}^N\omega^R_j\tilde{{F}}^{\EC,\eqInd,R}_{0j}-\sum_{j=0}^N\omega_j^L\tilde{{F}}^{\EC,\eqInd,L}_{Nj},\label{ConCondbefore}\\
\IEnt&=\sum_{j=0}^N\omega^R_j\left(\sum\limits_{\eqInd=1}^M{V}_{0j}^{\eqInd,R}\tilde{{F}}^{\EC,\eqInd,R}_{0j}-\tilde{\Psi}^{R,f}_{0j}\right)-\sum_{j=0}^N\omega^L_j\left(\sum\limits_{\eqInd=1}^M{V}_{Nj}^{\eqInd,L}\tilde{{F}}^{\EC,\eqInd,L}_{Nj}-\tilde{\Psi}^{L,f}_{Nj}\right)\label{ICcondbefore},
\end{align}
where the subscript $L$ and $R$ refer to the left and right element, respectively. Here, $\ICon$ and $\IEnt$ approximate the integral of $\bm{u}_t$ and $\dEnt$ on a single interface. In order to derive a discretely conservative scheme, meaning that \eqref{BothConditionsMom} and \eqref{BothConditionsEntCon} hold, we need to derive numerical surface fluxes so that
\begin{equation}
\begin{split}
\ICon&=0,\\
\IEnt&=0,
\end{split}
\end{equation}
is satisfied.

Here, since we consider conforming interfaces, it is assumed that $\Delta y:=\Delta y_R=\Delta y_L$. Also, as we focus on a one dimensional interface (first component of $\tilde{{F}}^{\EC,\eqInd}, {V}^{\eqInd,R}$ and $\tilde{\Psi}^{f}$ are fixed), we set
\begin{equation}
\begin{split}
\tilde{\bm{F}}^{\EC,\eqInd,L}&:=(\tilde{{F}}_{N_L 0}^{\EC,\eqInd,L},\dots,\tilde{{F}}_{N_LN_L}^{\EC,\eqInd,L})^T,\\
\tilde{\bm{F}}^{\EC,\eqInd,R}&:=(\tilde{{F}}_{00}^{\EC,\eqInd,R},\dots,\tilde{{F}}_{0N_R}^{\EC,\eqInd,R})^T,
\end{split}
\end{equation}
and the same for $\tilde{\bm\Psi}^{L,f},\tilde{\bm\Psi}^{R,f},\bm V^{\eqInd,L},\bm V^{\eqInd,R}$ respectively.

Furthermore, we introduce the notation of the discrete inner product to approximate the $L_2$ inner product. Assume we have two continuous functions $a(x),b(x)$ with their discrete evaluation $\bm A,\bm B$ on $[-1,1]$, then
\begin{equation}\label{eq:disInProd}
\ipC{\bm A}{\bm B} := \bm A^T\mass \bm B\approx \int_{-1}^1a(x)b(x)~dx=:\ip{a}{L_2}{b}.
\end{equation}
Based on the inner product notation we can rewrite \eqref{ConCondbefore} and \eqref{ICcondbefore} by
\begin{align}
\ICon& = \ipR{\OneR}{\tilde{ \bm{F}}^{\EC,\eqInd,R}}-\ipL{\OneL}{\tilde{\bm{F}}^{\EC,\eqInd,L}},\label{ConCond}\\
\IEnt& = \sum\limits_{\eqInd=1}^M\ipR{ \bm{V}^{\eqInd,R}}{\tilde{ \bm{F}}^{\EC,\eqInd,R}}-\ipR{\OneR}{\tilde{\bm{\Psi}}^{R,f}}-\sum\limits_{\eqInd=1}^M\ipL{ \bm{V}^{\eqInd,L}}{\tilde{\bm{F}}^{\EC,\eqInd,L}}+\ipL{\OneL}{\tilde{\bm{\Psi}}^{L,f}}.\label{ICcond},
\end{align}
where $\OneL,\OneR$ are vectors of ones with size ${N_L+1}$ and $N_R+1$, respectively. The choice of the numerical flux depends on the nodal distribution in each element. Here, we differ between conforming and non-conforming nodal distributions, which is done in the next section.

\subsubsection{Conforming Nodal Distribution}

We first provide a brief overview on the entropy conservative properties of the conforming DGSEM \eqref{eq:ECDG}. This is straightforward in the conforming case and we use this discussion to introduce notation which is necessary for the non-conforming proofs presented later. That is, the nodal distributions in each element are identical and there are no hanging nodes in the mesh. For a conforming approximation it is possible to have a point-to-point transfer of solution information at interfaces because the mass matrix, polynomial order and numerical flux ``match''
\begin{equation}
\label{ConfHelp}
\begin{split}
\mass&:=\massR=\massL,\\
N&:=N_R=N_L,\\
\tilde{\bm{F}}^{\EC,\eqInd}&:=\tilde{\bm{F}}^{\EC,\eqInd,R}=\tilde{\bm{F}}^{\EC,\eqInd,L}.
\end{split}
\end{equation}
Primary and entropy conservation can be achieved by choosing an entropy conservative numerical flux function as shown by Fisher et al. \cite{Fisher2013}. We include the proof for completeness and recast it in our notation in Lemma \ref{Thm:Conf}.\\

\begin{lem}\label{Thm:Conf}
Assume we have an entropy conservative numerical flux function, $\tilde{\bm{F}}^{\EC}$, that satisfies \eqref{eq:entCondition}, then the split form DGSEM is primary and entropy conservative for the DGSEM \eqref{eq:splitDG} by setting the numerical volume and surface fluxes to be $\tilde{\bm{F}}^{\#}=\tilde{\bm{F}}^{*}:=\tilde{\bm{F}}^{\EC}$.
\end{lem}
\begin{proof}
Primary conservation can be shown easily by inserting \eqref{ConfHelp} in \eqref{ConCond}:
\begin{equation}
\ICon:=\ipC{\One}{\tilde{\bm{F}}^{\EC,\eqInd}}-\ipC{\One}{\tilde{\bm{F}}^{\EC,\eqInd}}=0.
\end{equation}
For entropy conservation we analyze \eqref{ICcond}
\begin{equation}
\label{ICcondpause}
\IEnt =  \sum\limits_{\eqInd=1}^M\ipC{\bm{V}^{\eqInd,R}}{\tilde{\bm{F}^{\EC,\eqInd}}}- \sum\limits_{\eqInd=1}^M\ipC{\bm{V}^{\eqInd,L}}{\tilde{\bm{F}}^{\EC,\eqInd}}-\ipC{\One}{\tilde{\bm{\Psi}}^{f,R}}+\ipC{\One}{\tilde{\bm{\Psi}}^{f,L}}.
\end{equation}
For the discrete inner product it holds
\begin{equation}
\ipC{\bm A}{\bm B}=\ipC{\One}{ \bm A\circ \bm B},
\end{equation}
where $\circ$ denotes the Hadamard product for matrices. Then \eqref{ICcondpause} is rearranged to become
\begin{align}
\IEnt&= \ipC{\One}{\sum\limits_{\eqInd=1}^M\bm{V}^{\eqInd,R}\circ\tilde{\bm{F}}^{\EC,\eqInd}}-\ipC{\One}{\sum\limits_{\eqInd=1}^M\bm{V}^{\eqInd,L}\circ\tilde{\bm{{F}}^{\EC,\eqInd}}}
-\ipC{\One}{\tilde{\bm{ \Psi}}^{f,R}}+\ipC{\One}{\tilde{\bm{\Psi}}^{f,L}}\notag\\
&=\ipC{\One}{\sum\limits_{\eqInd=1}^M(\bm{V}^{\eqInd,R}-\bm{V}^{\eqInd,L})\circ\tilde{\bm{F}}^{\EC,\eqInd}-(\tilde{\bm\Psi}^{f,R}-\tilde{\bm\Psi}^{f,L})}\\
&=\ipC{\One}{\bm\Phi}\notag,
\end{align}
where $\bm\Phi:= \sum\limits_{\eqInd=1}^M(\bm{V}^{\eqInd,R}-\bm{V}^{\eqInd,L})\circ\tilde{\bm{F}}^{\EC,\eqInd}-(\tilde{\bm\Psi}^{f,R}-\tilde{\bm\Psi}^{f,L})$. By analyzing a single component of $\bm\Phi$ we find
\begin{equation}
\Phi_i= \sum\limits_{\eqInd=1}^M\left({V}^{\eqInd,R}_{i0}- {V}^{\eqInd,L}_{iN}\right)  \tilde{{F}}^{\EC,\eqInd}(\bm{U}_{iN}^L,\bm{U}_{i0}^R)-( \tilde\Psi_{i0}^{f,R}- \tilde\Psi_{iN}^{f,L})=0,
\end{equation}
due to \eqref{eq:entCondition}. So
\begin{equation}
\IEnt=\ipC{\One}{\bm 0}=0,
\end{equation}
which leads to an entropy conservative nodal DG scheme.
\end{proof}
How to modify the entropy conservative numerical flux with dissipation to ensure that the scheme is entropy stable is described later in Sec. \ref{sec:Dissipation}. For now, we address the issue of $h/p$ refinement where non-conforming meshes may contain differing nodal distributions or hanging nodes. To do so, we consider the entropy conservative fluxes in a modified way. Namely, the projection procedure of the standard mortar method is augmented in the next sections to guarantee the entropic properties of the numerical approximation.

%

\subsubsection{Non-Conforming Nodal Distribution}\label{subsec: p}

In this section we focus on a discretization, where the nodes do not coincide ($p$-refinement), see Fig. \ref{fig:NonConfMesh}(a). As such, we introduce projection operators
\begin{equation}
\mat{P}_{L2R}\in\mathbb{R}^{(N_R+1)\times (N_L+1)},\quad \mat{P}_{R2L}\in\mathbb{R}^{(N_L+1)\times (N_R+1)}.
\end{equation}
In particular, the solution on either element is always moved to its neighbor where the entropy conservative numerical flux is computed. In a sense, this means we ``hide'' the mortar used to cement the two elements together in the non-conforming approximation. This presentation is motivated to simplify the discussion. The mortars are a useful analytical tool to describe the idea of a non-conforming DG method, but in a practical implementation they can be removed with a careful construction of the projection operators.

Here $\mat P_{L2R}$ denotes the projection from the left element to the right element, whereas $\mat P_{R2L}$ denotes the projection from the right element to the left. In the approximation we have two solution polynomials $\bm{\mathfrak{p}}_L$ and $\bm{\mathfrak{p}}_R$ evaluated at the corresponding interfaces of each element. The numerical approximation is primary and entropy conservative provided both \eqref{ConCond} and \eqref{ICcond} are zero. However, the subtractions involve two discrete inner products with differing polynomial order between the left and right elements. Therefore, we require projection operators that move information from the left node distribution to the right and vice versa. As such, in discrete inner product notation, the projections must satisfy \cite{Mattsson2010b}
\begin{equation}\label{ProjMot}
\ipL{\bm{\mathfrak{p}}_L}{\mat{P}_{R2L}\bm{\mathfrak{p}}_R}=\ipR{\mat{P}_{L2R}\bm{\mathfrak{p}}_L}{\bm{\mathfrak{p}}_R}\;
\Leftrightarrow\;
\bm{\mathfrak{p}}_L^T\massL\mat{P}_{R2L}\bm{\mathfrak{p}}_R=\bm{\mathfrak{p}}_L^T\mat{P}_{L2R}^T\massR\bm{\mathfrak{p}}_R.
\end{equation}
As the polynomials in \eqref{ProjMot} are arbitrary, we set the projection operators to be \textit{$\mass$-compatible}, meaning
\begin{equation}\label{7}
\begin{split}
\mat{P}_{R2L}^T\massL&=\massR\mat{P}_{L2R},
\end{split}
\end{equation}
which is the same constraint considered in \cite{Carpenter2016,Friedrich2016,Kozdon2016,Mattsson2010b,Parsani2016}. Non-conforming methods with DG operators have been derived by Kopriva \cite{Kopriva2002} on LGL nodes, which imply a diagonal SBP norm. The construction of the projection operators is motivated by a discrete $L_2$ projection over Lagrange polynomials and can be found in Appendix \ref{sec:App B}.

The conditions for primary conservation \eqref{ConCond} and entropy conservation \eqref{ICcond} can be directly adapted from the conforming case. Before proving total conservation, we first introduce the operator $\EE$ to simplify the upcoming proof of Theorem \ref{thm: p} and to make it more compact. The operator $\EE$ extracts the diagonal of a matrix:
\begin{equation}\label{eq:EE}
\EE\begin{pmatrix} a_{11} & \hdots & \hdots & a_{1n} \\ \vdots & \ddots &  & \vdots \\ \vdots &  & \ddots & \vdots \\ a_{n1} & \hdots & \hdots & a_{nn} \end{pmatrix}= \begin{pmatrix} a_{11}\\ a_{22} \\ \vdots \\ a_{n-1,n-1} \\ a_{nn} \end{pmatrix},
\end{equation}
and has the following property. \\

\begin{lem}\label{Lemma}
 Given a vector $\bm a\in\mathbb{R}^{N_L+1}$, a diagonal matrix $\mat A=\diag(\bm a)\in\mathbb{R}^{(N_L+1)\times(N_L+1)}$ and a dense rectangular matrix $\mat{B}\in\mathbb{R}^{(N_L+1)\times(N_R+1)}$, then
\begin{equation}
\ipL{\bm a}{\EE(\mat P_{R2L}\mat B^T)}=\ipR{\OneR}{\EE(\mat P_{L2R}\mat A\mat B)}.
\end{equation}
\end{lem}
\begin{proof}
\begin{equation}
\begin{split}
\ipL{\bm a}{\EE(\mat P_{R2L}\mat B^T)}&=\OneLT\mat A \massL \EE(  \mat P_{R2L}\mat B^T)=\OneLT \EE(\mat A \underbrace{\massL  \mat P_{R2L}}_{
\stackrel{\eqref{7}}{=}\mat P_{L2R}^T  \massR}\mat B^T),\\
&=\OneLT \EE(\underbrace{\mat A \mat P_{L2R}^T  }_{=:\tilde{\mat  A}}\underbrace{\massR\mat B^T}_{=:\tilde{\mat  B}})=\OneLT \EE(\tilde{\mat  A}\tilde{\mat  B}),
\end{split}
\end{equation}
since $\mat A$ and the norm matrix $\massL$ are diagonal matrices they are free to move inside the extraction operator \eqref{eq:EE} and $\tilde{\mat A}\in\mathbb{R}^{(N_L+1)\times(N_R+1)}, \tilde{\mat B}\in\mathbb{R}^{(N_R+1)\times(N_L+1)}$. Note, that
\begin{equation*}
\OneLT \EE(\tilde{\mat  A}\mat{ \tilde B})=\sum_{i=0}^{N_L}1\sum_{j=0}^{N_R} \tilde A_{ij} \tilde B_{ji}=\sum_{j=0}^{N_R}1\sum_{i=0}^{N_L} \tilde B_{ji} \tilde A_{ij}=\OneRT\EE(\tilde{\mat  B}\mat{ \tilde A}).
\end{equation*}
By replacing $\tilde{\mat  A}, \tilde{\mat  B}$ we get
\begin{equation}
\begin{split}
\ipL{\bm a}{\EE(\mat P_{R2L}\mat B^T)}&=\OneRT\EE(\massR\mat B^T\mat A \mat P_{L2R}^T)=\OneRT\massR\EE(\mat P_{L2R}\mat A\mat B),\\
&=\ipR{\OneR}{\EE(\mat P_{L2R}\mat A\mat B)},
\end{split}
\end{equation}
because $\EE(\mat W)=\EE(\mat W^T)$ for any square matrices $\mat W$.
\end{proof}

Furthermore, we introduce the following matrices
\begin{equation}\label{def}
\begin{split}
[\FstarLRK]_{ij}&=\tilde{{F}}^{\EC,\eqInd}(\bm{U}^L_{iN_L},\bm{U}^R_{j0}),\\
[\tilde{\mat\Psi}^{f,L}]_{ij}&=\tilde\Psi^f(\bm{U}^L_{iN_L}), \\
[\tilde{\mat\Psi}^{f,R}]_{ij}&=\tilde\Psi^f(\bm{U}^R_{j0}),
\end{split}
\end{equation}
for $i=0,\dots,N_L$, $j=0,\dots,N_R$, where $N_L$ and $N_R$ denote the number of nodes in one dimension in left and right element and $\eqInd=1,\dots,M$. Here, $\tilde{{F}}^{\EC,\eqInd}$ denotes a flux satisfying \eqref{eq:entCondition}. We note that the matrices containing the entropy flux potential are constant along rows or columns respectively and that for the non-conforming case $N_L\neq N_R$, so all matrices in \eqref{def} are rectangular. \\
With the operator $\EE$, Lemma \ref{Lemma} and \eqref{def}, it is possible to construct an entropy conservative scheme for non-conforming, non-linear problems.

\begin{thm} \label{thm: p}
Assume we have an entropy conservative numerical flux $\tilde{\bm{F}}^{\EC}$ from a conforming discretization satisfying the condition \eqref{eq:entCondition}. The fluxes
\begin{align}
 \tilde{{F}}^{\EC,\eqInd,L}_{i}&:=\sum_{j=0}^{N_R}\left[\mat{P}_{R2L}\right]_{ij}\left[\left(\FstarLRK\right)^T\right]_{ji},\quad i = 0,\ldots,N_L,\\
 \tilde{{F}}^{\EC,\eqInd,R}_{j}&:=\sum_{i=0}^{N_L}\left[\mat{P}_{L2R}\right]_{ji}[\FstarLRK]_{ij},\quad j = 0,\ldots,N_R,
\end{align}
or in a more compact matrix-vector notation
\begin{align}
 \tilde{\bm{F}}^{\EC,\eqInd,L}&:=\EE\left(\mat{P}_{R2L}\FstarLRKT\right),\label{8.1}\\
\tilde {\bm{F}}^{\EC,\eqInd,R}&:=\EE\left(\mat{P}_{L2R}\FstarLRK\right),\label{8.2}
\end{align}
are primary and entropy conservative for non-conforming nodal distributions.
\end{thm}
\begin{proof} First, we prove primary conservation by including \eqref{8.1} and \eqref{8.2} in \eqref{ConCond}
\begin{equation}\label{eq:IUT}
\begin{split}
\ICon:=&\ipR{\OneR}{\EE\left(  \mat{P}_{L2R}\FstarLRK\right)}-\ipL{\OneL}{\EE\left(  \mat{P}_{R2L}\FstarLRKT\right)}.\\
\end{split}
\end{equation}
We apply the result of Lemma \ref{Lemma} to the last term of \eqref{eq:IUT} with $\bm a = \OneL$ and $\mat B = \FstarLRK$ to get the conservation for the primary variables
\begin{equation}
\ICon=\ipR{\OneR}{\EE\left(  \mat{P}_{L2R}\FstarLRK\right)}-\ipR{\OneR}{\EE\left(  \mat{P}_{L2R}\FstarLRK\right)}=0.
\end{equation}
Next, we show that the discretization is entropy conservative. To do so, we substitue the fluxes \eqref{8.1} and \eqref{8.2} in  \eqref{ICcond}.
\begin{equation}\label{ThreeParts}
\begin{aligned}
\IEnt:=&\sum\limits_{\eqInd=1}^M \ipR{\evR}{\EE\left(  \mat{P}_{L2R}\FstarLRK\right)}-\ipR{\OneR}{\tilde{\bm{\Psi}}^{f,R}}\\
&-\sum\limits_{\eqInd=1}^M\ipL{\evL}{\EE\left(  \mat{P}_{R2L}\FstarLRKT\right)} + \ipL{\OneL}{\tilde{\bm{\Psi}}^{f,L}}.
\end{aligned}
\end{equation}
We divide \eqref{ThreeParts} into three pieces to simplify the analysis.
\begin{equation}\label{final}
\begin{aligned}
\IEnt=&\underbrace{\sum_{\eqInd=1}^M\ipR{\evR}{\EE\left(  \mat{P}_{L2R}\FstarLRK\right)}}_{=(I)}-\underbrace{\sum\limits_{\eqInd=1}^M\ipL{\evL}{\EE\left(  \mat{P}_{R2L}\FstarLRKT\right)}}_{=(II)}\\
&-\left(\underbrace{\ipR{\OneR}{\tilde{\bm{\Psi}}^{f,R}} - \ipL{\OneL}{\tilde{\bm{\Psi}}^{f,L}}}_{=(III)}\right).
\end{aligned}
\end{equation}
For $(I)$ we see that
\begin{align}
(I)= \sum\limits_{\eqInd=1}^M\ipR{\evR}{\EE\left(\mat{P}_{L2R}\FstarLRK\right)}= \sum\limits_{\eqInd=1}^M\ipR{\OneR}{\bm{V}^{\eqInd,R}\circ\EE\left(\mat{P}_{L2R}\FstarLRK\right)}.
\end{align}
By introducing $\mat{V}^{\eqInd,R}:=\mathrm{diag}(\bm V^{\eqInd,R})$ we can shift the entropy variables inside the $\EE$ operator and obtain
\begin{align}
(I) = \sum\limits_{\eqInd=1}^M\ipR{\OneR}{\EE\left(\mat{V}^{\eqInd,R}\mat{P}_{L2R}\tilde{\mat{F}}^{\EC,\eqInd}_{L,R}\right)}= \ipR{\OneR}{\EE\left(\mat{P}_{L2R}\left(\sum\limits_{\eqInd=1}^M\mat{V}^{\eqInd,R}\FstarLRK\right)\right)},
\end{align}
because $\EE(\mat A\mat B)=\EE(\mat B\mat A)$ for square matrices $\mat A,\mat B$.

Considering the second term (II) of \eqref{final}
\begin{align}
(II)=& \sum\limits_{\eqInd=1}^M\ipL{\evL}{\EE\left(  \mat{P}_{R2L}\FstarLRKT\right)},
\end{align}
and applying Lemma \ref{Lemma} with $\bm a = \evL, \mat A = \mat{V}^{\eqInd,L}$, and $\mat B = \FstarLRK$ gives
\begin{align}
(II)=\sum\limits_{\eqInd=1}^M\ipR{\OneR}{\EE\left(\mat{P}_{L2R}\mat{V}^{\eqInd,L}\FstarLRK  \right)}=\ipR{\OneR}{\EE\left(\mat{P}_{L2R}\left(\sum\limits_{\eqInd=1}^M\mat{V}^{\eqInd,L}\FstarLRK\right)  \right)}.
\end{align}

Last, we analyze term $(III)$,
\begin{equation}\label{(III)}
(III)=\ipR{\OneR}{\tilde{\bm{\Psi}}^{f,R}}-\ipL{\OneL}{\tilde{\bm{\Psi}}^{f,L}}.
\end{equation}
For this analysis we rewrite $(III)$ in terms of the matrices $\tilde{\mat{\Psi}}^{f,L}, \tilde{\mat{\Psi}}^{f,R}$. Note, that each column of $\tilde{\mat{\Psi}}^{f,R}$ and each row of $\tilde{\mat{\Psi}}^{f,L}$ remain constant. The projection operators are exact for a constant state, i.e. $\mat{P}_{R2L}\OneR=\OneL$ and $\mat{P}_{L2R}\OneL=\OneR$. Hence, we define
\begin{align}
\tilde{\bm{\Psi}}^{f,R}&=\EE\left(  \mat{P}_{L2R}\tilde{\mat{\Psi}}^{f,R}\right),\\
\tilde{\bm{\Psi}}^{f,L}&=\EE\left(  \mat{P}_{R2L}\left(\tilde{\mat{\Psi}}^{f,L}\right)^{\!T}\right).
\end{align}
Substituting the above definitions in \eqref{(III)} we arrive at
\begin{align}
(III)=\ipR{\OneR}{\EE\left(  \mat{P}_{L2R}\tilde{\mat{\Psi}}^{f,R}\right)}-\ipL{\OneL}{\EE\left(  \mat{P}_{R2L}\left(\tilde{\mat{\Psi}}^{f,L}\right)^{\!T}\right)}.
\end{align}
Again applying Lemma \ref{Lemma} (where $\bm a=\OneL, \mat B = \tilde{\mat{\Psi}}^{f,L}$) yields
\begin{align}
(III)=&\ipR{\OneR}{\EE\left(\mat{P}_{L2R}\tilde{\mat{\Psi}}^{f,R}  \right)}-\ipR{\OneR}{\EE\left(\mat{P}_{L2R}\tilde{\mat{\Psi}}^{f,L}  \right)},\\
=&\ipR{\OneR}{\EE\left(\mat{P}_{L2R}\left(\tilde{\mat{\Psi}}^{f,R}- \tilde{\mat{\Psi}}^{f,L} \right)\right)}.
\end{align}
Substituting $(I),(II),(III)$ in \eqref{final} we have rewritten the entropy update to be
\begin{equation}\label{final1}
\begin{split}
\IEnt=&\ipR{\OneR}{\EE\left(\mat{P}_{L2R}\left(\sum\limits_{\eqInd=1}^M\mat{V}^{\eqInd,R}\FstarLRK-\sum\limits_{\eqInd=1}^M\mat{V}^{\eqInd,L}\FstarLRK-\left(\tilde{\mat{\Psi}}^{f,R}- \tilde{\mat{\Psi}}^{f,L}\right)\right)\right)},\\
=&\ipR{\OneR}{\EE\left(\mat{P}_{L2R}\mat{ \tilde S}  \right)},
\end{split}
\end{equation}
with
\begin{equation}
\tilde{\mat S}:=\sum\limits_{\eqInd=1}^M\mat{V}^{\eqInd,R}\FstarLRK-\sum\limits_{\eqInd=1}^M\mat{V}^{\eqInd,L}\FstarLRK-\left(\tilde{\mat{\Psi}}^{f,R}- \tilde{\mat{\Psi}}^{f,L}\right).
\end{equation}
Next, we analyze a single component of $\mat{ \tilde S}$. Let $i=0,\dots N_L$ and $j=0,\dots N_R$, then
\begin{equation}\label{almostthere}
[\mat{ \tilde S}]_{ij}=\sum\limits_{\eqInd=1}^M\left(V^{q,R}_{j0}- V^{q,L}_{iN_L}\right)\tilde{{F}}^{\EC,\eqInd}(\bm{U}_{iN_L}^L,\bm{U}_{j0}^R)-(\tilde{\Psi}_{j0}^{f,R}-\tilde{\Psi}_{iN_L}^{f,L}).
\end{equation}
Since the entropy conservative fluxes is contained in \eqref{almostthere} and due to \eqref{eq:entCondition} we obtain
\begin{equation}
[\mat{ \tilde S}]_{ij}=0.
\end{equation}
Inserting this result in \eqref{final1} we arrive at
\begin{equation}
\begin{split}
\IEnt=&\ipR{\OneR}{\EE\left(\mat{P}_{L2R}\mat 0  \right)}=0.
\end{split}
\end{equation}
Therefore, $\IEnt$ is zero for $ \tilde{{F}}^{\EC,\eqInd,L}:=\EE\left(\mat{P}_{R2L}\FstarLRKT\right)$ and $ \tilde{{F}}^{\EC,\eqInd,R}:=\EE\left(\mat{P}_{L2R}\FstarLRK\right)$.
 \end{proof}

Note, that this proof is for general for any hyperbolic PDE with physical fluxes $f,g$ where we have an entropy. Based on this proof, we can construct entropy conservative schemes with algebraic non-conforming discretizations ($p$ refinement). To introduce additional flexibility, we next consider geometric non-conforming discretizations where the interfaces may not coincide ($h$ refinement).

\subsection{Non-Conforming Interfaces with Hanging Nodes}\label{sec:h}

In Sec. \ref{subsec: p} we created numerical fluxes for elements with a coinciding interface but differing polynomial orders. As such, each numerical interface flux only depends on one neighboring element. For example the numerical flux $\tilde{\bm{F}}^{\EC,R}$ in \eqref{8.2} only contained the projection operator $\mat P_{L2R}$, so it only depended on one neighboring element $L$. This was acceptable if the interfaces had no hanging nodes, however for the more general case of $h$ refinement as in Fig. \ref{hangingnodes} the interface coupling requires addressing contributions from many elements.
\begin{figure}[!ht]
\begin{center}
\begin{tikzpicture}[scale=0.5]

\color[rgb]{1,0,0}
\draw[very thick](0,-4) -- (0,4) ;
\color[rgb]{0,0,0}
\draw[very thick] (0,4) -- (-8,4)-- (-8,-4) -- (0,-4)  ;
\node at (-4,3) {$L_1$};
\draw[very thick] (-8,2) -- (0,2) ;
\node at (0,2) {X};
\node at (-4,1) {$L_2$};
\draw[very thick] (-8,0) -- (0,0) ;
\node at (0,0) {X};
\node at (-4,-1) {\vdots};
\draw[very thick] (-8,-2) -- (0,-2) ;
\node at (0,-2) {X};
\node at (-4,-3) {$L_E$};
\draw[<->] (-8.4,2.2) -- (-8.4,3.8);
\node at (-9.4,-3) {$\Delta_{L_E}$} ;
\draw[<->] (-8.4,0.2) -- (-8.4,1.8);
\node at (-9.4,1) {$\Delta_{L_2}$} ;
\draw[<->] (-8.4,-2.2) -- (-8.4,-3.8);
\node at (-9.4,3) {$\Delta_{L_1}$} ;
\draw[very thick] (8,-4) -- (0,-4) ;
\draw[very thick] (0,4) -- (8,4) ;
\draw[very thick] (8,4) -- (8,-4) ;
\draw[<->] (8.4,-4) -- (8.4,4);
\node at (9.4,0) {$\Delta_R$} ;
\node at (4,0) {$R$};

\end{tikzpicture}
\end{center}
\caption{$h$ refinement with hanging nodes X}
\label{hangingnodes}
\end{figure}
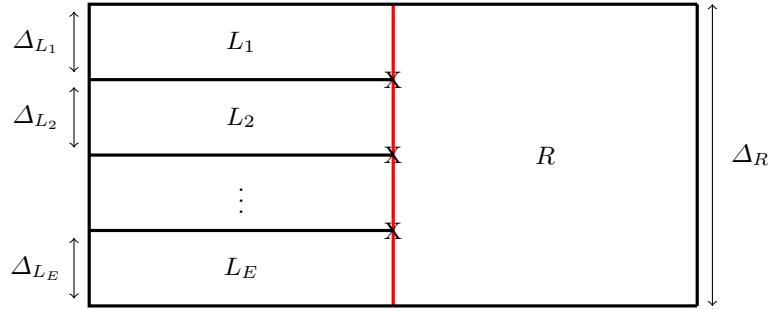

Throughout this section we will focus on discrete meshes as in Fig. \ref{hangingnodes}. For the $h$ refinement analysis we adapt the results derived in the previous section, where the interfaces coincide. Therefore, we consider all left elements as if they are \emph{\textit{one}} element $L=\bigcup_{i=1}^EL_i$. Again, this procedure ``hides'' the mortars within the new projection operators. Thus, we see that each sub-element $L$ has a conforming interface with element $R$ (red line) and has the nodes of the elements $L_i$ on the red lined interface
\begin{equation}
\eta^L=(\eta^{L_1}_0,\dots,\eta^{L_1}_{N_{L_1}},\dots,\eta^{L_E}_0,\dots,\eta^{L_E}_{N_{L_E}})^T,
\end{equation}
where $\eta^{L_i}$ denotes the vertical nodes of the element $L_i$. For element $L$ and element $R$ the projection operators need to satisfy the $\mass$-compatibility condition \eqref{7}:
\begin{equation}\label{hproj}
\mat{P}_{L2R}^T\mass_R=\mass_L\mat{P}_{R2L},
\end{equation}
where
\begin{equation}
\mass_L=\frac{1}{\Delta_R}\begin{pmatrix}
\Delta_{L_1}\mass_{L_1} & & \\
& \ddots & \\
& & \Delta_{L_E}\mass_{L_E}
\end{pmatrix},
\end{equation}
where $\Delta$ denotes the height of an element.
We can interpret the ``large'' projection operators into parts that contribute from/to each of the smaller elements with the following structure
\begin{equation}\label{construct}
\mat P_{L2R}=\begin{bmatrix}
\mat P_{L_12R}&\dots & \mat P_{L_E2R}
\end{bmatrix},
\end{equation}
and
 \begin{equation}
\mat P_{R2L}=\begin{bmatrix}
\mat P_{R2L_1}\\ \vdots \\ \mat P_{R2L_E}
\end{bmatrix}.
\end{equation}
With this new notation we adapt the $\mass$-compatibility condition \eqref{hproj} to become
\begin{equation}\label{hproj2}
\Delta_{L_i} \mat{P}_{R2L_i}^T\massLi=\Delta_R\massR\mat{P}_{L_i2R},\quad i = 1,\ldots,E.
\end{equation}

As in Sec. \ref{sec:p} we choose the numerical surface fluxes so that the scheme is primary and entropy conservative. We note that for the $h$ non-conforming case (just like $p$ non-conforming) the result of Lemma \ref{thm1} is still valid. That is, the volume contributions have no effect on the non-conforming approximation. Only a careful definition of the interface coupling is needed to construct an entropy stable non-conforming DG approximation. Therefore, we analyze all terms which are related to the interface connecting $L_1,\dots,L_E$ and $R$. Similar to \eqref{ConCond} and \eqref{ICcond}, we arrive at the following terms
\begin{align}
\ICon=&\ipR{\OneR}{\tilde{\bm{F}}^{\EC,\eqInd,R}}-\sum_{i=1}^E\ipLi{\OneLi}{\tilde{\bm{F}}^{\EC,\eqInd,L_i}},\label{ConCond2}\\
\IEnt=&\sum\limits_{\eqInd=1}^M\ipR{\evR}{\tilde{\bm{F}}^{\EC,\eqInd,R}}-\ipR{\OneR}{\tilde{\mat{\Psi}}^{R,f}}\notag-\sum_{i=1}^E\left(\sum\limits_{\eqInd=1}^M\ipLi{\bm{V}^{\eqInd,L_i}}{\tilde{\bm{F}}^{\EC,\eqInd,L_i}}-\ipLi{\OneLi}{\tilde{\mat{\Psi}}^{L_i,f}}\right),\label{ICcond2}
\end{align}
which need to be zero to obtain a discretely primary and entropy conservative scheme.\\

\begin{cor}\label{cor}
 Given a set of projection operators that satisfy \eqref{hproj2}, we can prove analogously to Theorem \ref{thm: p} that the fluxes
\begin{equation}
 \tilde{{F}}^{\EC,\eqInd,R}:=\EE\left(\mat{P}_{L2R}\FstarLRK\right)
=\sum_{i=1}^E\EE\left(\mat{P}_{L_i2R}\FstarLiRK\right)\label{ECfluxR},
\end{equation}
and
\begin{equation}
 \tilde{{F}}^{\EC,\eqInd,L_i}:=\EE\left(\mat{P}_{R2L_i}\FstarLiRKT\right)\label{ECfluxL},
\end{equation}
for $i=1,\dots,E$ lead to primary and entropy conservative schemes.
\end{cor}
\begin{proof}
This result requires a straightforward modification of the result from Lemma \ref{Lemma}
\begin{equation}
\Delta_{L_i}\ipL{\bm a}{\EE(\mat P_{R2L_i}\mat B^T)}=\Delta_R\ipR{\OneR}{\EE(\mat P_{L_i2R}\mat A\mat B)}.
\end{equation}
Now, the proof follows identical steps as given for Lemma \ref{thm1}, but now keeping track of the adjustable element sizes.
\end{proof}

\subsection{Including Dissipation within the Numerical Surface Flux}\label{sec:Dissipation}

In Sec. \ref{sec:p} and Sec. \ref{sec:h} we derived primary and entropy conservative schemes for non-linear problems on non-conforming meshes with $h/p$ refinement. From these results, we can include interface dissipation and arrive at an entropy stable discretization for an arbitrary non-conforming rectangular mesh.

While conservation laws are entropy conservative for smooth solutions, discontinuities in the form of shocks can develop in finite time for non-linear problems despite smooth initial data. Considering shocks, the mathematical entropy should decay, which needs to be reflected within our numerical scheme. Thus, we will describe how to include interface dissipation which leads to an entropy stable scheme. We note that the numerical volume flux in \eqref{eq:ECDG} is still an entropy conservative flux which satisfies \eqref{eq:entCondition}.

We will focus on the general case, where we have differing nodal distributions as well as hanging nodes ($h$ refinement) as in Fig. \ref{hangingnodes}.  As in Sec. \ref{sec:h} we assume that the projection operators satisfy the compatibility condition \eqref{hproj2}.\\

\begin{thm}\label{thm: diss}
The scheme is primary conservative and entropy stable, for the following numerical surface fluxes.
\begin{equation}\label{ESfluxL}
\tilde{\bm{F}}^{\ES,\eqInd,L_i}= \tilde{\bm{F}}^{\EC,\eqInd,L_i}-\frac{\lambda}{2}\left(\mat P_{R2L_i} \bm{V}^{\eqInd,R}- \bm{V}^{\eqInd,L_i} \right)
\end{equation}
\begin{equation}\label{ESfluxR}
\tilde{\bm{F}}^{\ES,\eqInd,R}= \tilde{\bm{F}}^{\EC,\eqInd,R}-\frac{\lambda}{2}\sum\limits_{i=1}^E\mat P_{L_i2R}\left(\mat P_{R2L_i} \bm{V}^{\eqInd,R} - \bm{V}^{\eqInd,L_i}\right)
\end{equation}
where $\lambda>0$ is a scalar which controls the dissipation rate.
\end{thm}

\begin{proof}
By including dissipation we can prove primary conservation by substituting the new fluxes \eqref{ESfluxL} and \eqref{ESfluxR} into \eqref{ConCond2}
\begin{equation}
\resizebox{\hsize}{!}{$
\ICon=\ipR{\OneR}{\tilde{\bm{F}}^{\EC,\eqInd,R}-\frac{\lambda}{2}\sum\limits_{i=1}^E\mat P_{L_i2R}\left(\mat P_{R2L_i}\bm{V}^{\eqInd,R} - \bm{V}^{\eqInd,L_i}\right)}-\sum\limits_{i=1}^E\ipLi{\OneLi}{\tilde{ \bm{F}}^{\EC,\eqInd,L_i}-\frac{\lambda}{2}\left(\mat P_{R2L_i} \bm{V}^{\eqInd,R}- \bm{V}^{\eqInd,L_i} \right)}.$}
\end{equation}
Due to Corollary \ref{cor} we know that
\begin{equation}
\ipR{\OneR}{\tilde{{F}}^{\EC,\eqInd,R}}-\sum_{i=1}^E\ipLi{\OneLi}{\tilde{{F}}^{\EC,\eqInd,L_i}}=0,
\end{equation}
and we find that
\begin{equation}
\ICon=-\frac{\lambda\Delta_R}{4}\sum_{i=1}^E\OneRT\massR\mat P_{L_i2R}\left(\mat P_{R2L_i} \bm{V}^{\eqInd,R} - \bm{V}^{\eqInd,L_i}\right)+\sum_{i=1}^E\frac{\lambda\Delta_{L_i}}{4}\OneLiT\massLi\left(\mat P_{R2L_i} \bm{V}^{\eqInd,R}- \bm{V}^{\eqInd,L_i} \right).
\end{equation}
Due to \eqref{hproj2} we arrive at
\begin{equation}
\ICon=-\sum_{i=1}^E\frac{\lambda\Delta_{L_i}}{4}\OneRT\mat P_{R2L_i}^T\massLi\left(\mat P_{R2L_i}\bm{V}^{\eqInd,R}- \bm{V}^{\eqInd,L_i}\right)+\sum_{i=1}^E\frac{\lambda\Delta_{L_i}}{4}\OneLiT\massLi\left(\mat P_{R2L_i}\bm{V}^{\eqInd,R}- \bm{V}^{\eqInd,L_i} \right),
\end{equation}
assuming that $\mat P_{R2L_i}$ can project a constant exactly, meaning $\mat P_{R2L_i}\OneR=\OneLi$, it yields
\begin{equation}
\ICon=0,
\end{equation}
which leads to a primary conservative scheme.

To prove entropy stability we include \eqref{ESfluxL} and \eqref{ESfluxR} in \eqref{ICcond} and adapt the results from Corollary \ref{cor} to find that
\begin{equation}
\label{EntDiss}
\begin{split}
\IEnt=&-\sum\limits_{\eqInd=1}^M\frac{\Delta_R}{2}\ipR{\bm{V}^{\eqInd,R}}{\frac{\lambda}{2}\sum_{i=1}^E\mat P_{L_i2R}\left(\mat P_{R2L_i}\bm{V}^{\eqInd,R}- \bm{V}^{\eqInd,L_i}\right)}\\
&+\sum\limits_{\eqInd=1}^M\sum_{i=1}^E\frac{\Delta_{L_i}}{2}\ipLi{ \bm{V}^{\eqInd,L_i}}{\frac{\lambda}{2}\left(\mat P_{R2L_i}\bm{V}^{\eqInd,R}- \bm{V}^{\eqInd,L_i} \right)},\\
=&-\frac{\lambda\Delta_R}{4}\sum\limits_{\eqInd=1}^M\sum_{i=1}^E \evRT\massR\mat P_{L_i2R}\left(\mat P_{R2L_i}\bm{V}^{\eqInd,R}- \bm{V}^{\eqInd,L_i}\right)\\
&+\sum\limits_{\eqInd=1}^M\sum_{i=1}^E\frac{\lambda\Delta_{L_i}}{4} \left(\bm{V}^{\eqInd,L_i}\right)^{\!T}\!\massLi\left(\mat P_{R2L_i}\bm{V}^{\eqInd,R}- \bm{V}^{\eqInd,L_i}\right).
\end{split}
\end{equation}
Again, we apply the condition \eqref{hproj2} and obtain
\begin{equation}
\begin{split}
\IEnt=&-\sum_{i=1}^E\frac{\lambda\Delta_{L_i}}{4} \evRT\mat P_{R2L_i}^T\massLi\left(\mat P_{R2L_i} \bm{V}^{\eqInd,R} - \bm{V}^{\eqInd,L_i}\right)\\
&+\sum_{i=1}^E\frac{\lambda\Delta_{L_i}}{4} \left(\bm{V}^{\eqInd,L_i}\right)^{\!T}\!\massLi\left(\mat P_{R2L_i} \bm{V}^{\eqInd,R}- \bm{V}^{\eqInd,L_i} \right),\\
=&-\sum_{i=1}^E\frac{\lambda\Delta_{L_i}}{4}\left(\mat P_{R2L_i} \bm{V}^{\eqInd,R}- \bm{V}^{\eqInd,L_i}\right)^T\!\massLi\left(\mat P_{R2L_i} \bm{V}^{\eqInd,R} - \bm{V}^{\eqInd,L_i}\right)\le 0,
\end{split}
\end{equation}
since each $\massLi$ is a symmetric positive definite matrix and $\lambda\Delta_{L_i}>0$, so the non-conforming DG scheme is entropy stable.
\end{proof}
Note, that this proof also holds for deriving an entropy stable scheme for geometrically conforming interfaces but differing polynomial order ($p$ refinement) by setting $E=1$.

To summarize, we derived a primary conservative and entropy stable DGSEM for non-linear problems on general $h/p$ non-conforming meshes. Note, that all results hold for an arbitrary system of non-linear conservation laws as long as entropy conservative numerical fluxes exist that satisfy \eqref{eq:entCondition}.

\section{Numerical results}\label{sec:numResults}

For all numerical results presented in this work we considered the two dimensional Euler equations
\begin{equation}
\begin{pmatrix}
\rho\\
\rho u\\
\rho v\\
E
\end{pmatrix}_t
+\begin{pmatrix}
\rho u\\
\rho u^2+p\\
\rho u v\\
u(E+p)
\end{pmatrix}_x
+\begin{pmatrix}
\rho v\\
\rho u v\\
\rho v^2+p\\
v(E+p)
\end{pmatrix}_y
=\begin{pmatrix}
0\\
0\\
0\\
0
\end{pmatrix},
\end{equation}
on $\Omega\subset\mathbb{R}^2$ and $t\in[0,T]\subset\mathbb{R}^+$ with $E=\frac{1}{2}\rho(u^2+v^2)+\frac{p}{\gamma-1}$ and adiabatic coefficient $\gamma=1.4$.

The entropy conservative/stable non-conforming implementation of the DGSEM of the Euler equations uses the Ismail and Roe entropy conserving flux \cite{Ismail2009} in \eqref{8.1} and \eqref{8.2} for $p$ refinement and in \eqref{ECfluxR} and \eqref{ECfluxL} to apply $h$ refinement. 

We use an explicit time integration method to advance the approximate solution. In particular, we select the five-stage, fourth-order low-storage Runge-Kutta method of Carpenter and Kennedy \cite{Kennedy1994}. The explicit time step $\Delta t$ is selected by the CFL condition \cite{Friedrich2016}
\begin{equation}
\Delta t:=CFL\frac{\min_i\{\frac{\Delta x_i}{2}\frac{\Delta y_i}{2}\}}{\max_j\{N_j+1\}\lambda_{\mathrm{max}}},
\end{equation}
where $\Delta x_i$ and $\Delta y_i$ denote the width in $x$- and $y$-direction of the $i^{\mathrm{th}}$ element, $N_j$ denotes the number of nodes in one dimension of the $j^{\mathrm{th}}$ element, and $\lambda_{\mathrm{max}}$ denotes the maximum eigenvalue of the flux Jacobians over the whole domain.

In this section, we verify the experimental order of convergence as well as conservation of the primary quantities and entropy for the novel $h/p$ non-conforming DGSEM described in this work.

\subsection{Experimental Order of Convergence}

For our numerical convergence experiments, we set $T=1$ and $CFL=0.2$. We analyze the experimental order of convergence for an entropy stable flux. Therefore, we include dissipation to the baseline entropy conservative Ismail and Roe flux \cite{Ismail2009} at each element interface. In particular, we consider a local Lax-Friedrichs type dissipation term with $\bm z =n_1\bm u+n_2\bm v$, where $\bm n=(n_1,n_2)^T$ denotes the normal vector
\begin{equation}
\begin{aligned}
\lambda_L&=\max\left\{||\bm z_L+\bm c_L||_\infty, ||\bm z_L||_\infty,|| \bm z_L-\bm c_L||_\infty\right\},\\
\lambda_R&=\max\left\{||\bm z_R+\bm c_R||_\infty, ||\bm z_R||_\infty,|| \bm z_R-\bm c_R||_\infty\right\},\\
\lambda&=\frac{1}{2}\max\left\{\lambda_L,\lambda_R\right\},
\end{aligned}
\end{equation}
with $c=\sqrt{\frac{\gamma p}{\rho}}$. By including $\lambda$ in \eqref{ESfluxL} and \eqref{ESfluxR}.

For convergence studies we consider the isentropic vortex advection problem taken from \cite{Chen2016}. Here, we set the domain to be $\Omega=[0,10]\times[0,10]$. The initial conditions are
\begin{equation}
\bm{w}(x,y,0) \equiv \bm{w}_0(x,y)
=
\begin{pmatrix}
T^{\frac{1}{\gamma-1}}\\[0.1cm]
1-(y-5)\phi(r)\\[0.1cm]
1+(x-5)\phi(r)\\[0.1cm]
T^{\frac{\gamma}{\gamma-1}}\\[0.1cm]
\end{pmatrix},
\end{equation}
where we introduce the vector of primitive variables $\bm{w}(x,y,t) = (\rho,u,v,p)^T$ and
\begin{equation}
\begin{split}
r(x,y)=&\sqrt{(x-5)^2+(y-5)^2},\\[0.1cm]
T(x,y)=&1-\frac{\gamma-1}{2\gamma}\phi(r^2),\\[0.1cm]
\phi(r)=&\varepsilon e^{\alpha(1-r^2)},\\[0.1cm]
\end{split}
\end{equation}
with $\varepsilon=\frac{5}{2\pi}$ and $\alpha=0.5$. With these initial condition the vortex is advected along the diagonal of the domain. We impose Dirichlet boundary conditions using the exact solution which is easily determined to be
\begin{equation}
\bm{w}(x,y,t)
=
\bm{w}_0(x-t,y-t).
\end{equation}

To examine the convergence order for a $h/p$ non-conforming method we consider a general mesh setup that includes pure $p$ non-conforming interfaces, pure $h$ non-conforming interfaces and $h/p$ non-conforming interfaces. Therefore we define three element types $A, B, C$.
Here, the mesh is prescribed in the following way
\begin{center}
\begin{itemize}
\item Elements of type $A$ in $\Omega_1=[0,5]\times[0,10]$\\
\item Elements of type $B$ in $\Omega_1=[5,10]\times[0,5]$\\
\item Elements of type $C$ in $\Omega_1=[5,10]\times[5,10]$.
\end{itemize}
\end{center}
For each level of the convergence analysis, a single element is divided into four sub-elements. This mesh refinement strategy is sketched in Fig. \ref{mesh3}.

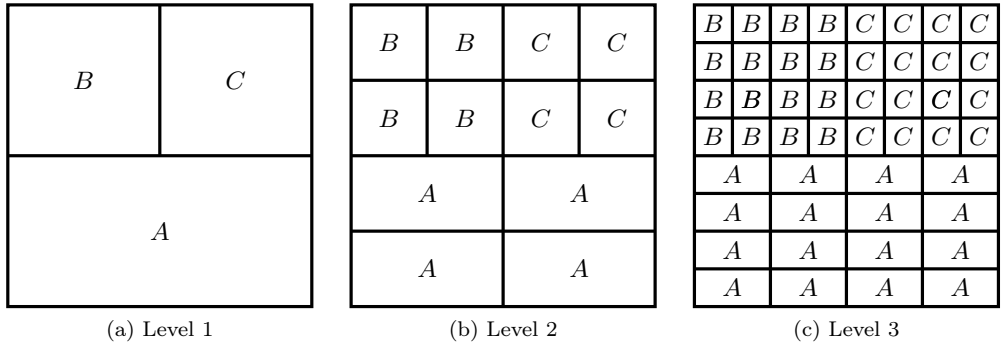
\begin{figure}[!ht]
\begin{center}
\subfloat[Level 1]
{
\begin{tikzpicture}[scale=0.25]
\draw[very thick] (-8,-8) -- (-8,8) -- (8,8) -- (8,-8) -- cycle ;
\draw[very thick] (-8,0) -- (8,0) ;
\draw[very thick] (0,0) -- (0,8) ;
\node at (0,-4) {$A$};
\node at (-4,4) {$B$};
\node at (4,4) {$C$};
\end{tikzpicture}
}\hspace{0.15cm}
\subfloat[Level 2]
{
\begin{tikzpicture}[scale=0.25]
\draw[very thick] (-8,-8) -- (-8,8) -- (8,8) -- (8,-8) -- cycle ;
\draw[very thick] (-8,0) -- (8,0) ;
\draw[very thick] (0,0) -- (0,-8) ;
\draw[very thick] (-8,-4) -- (8,-4) ;
\node at (4,-2) {$A$};
\node at (-4,-2) {$A$};
\node at (4,-6) {$A$};
\node at (-4,-6) {$A$};
\draw[very thick] (0,0) -- (0,8) ;
\draw[very thick] (-8,4) -- (0,4) ;
\draw[very thick] (-4,0) -- (-4,8) ;
\node at (-2,6) {$B$};
\node at (-6,6) {$B$};
\node at (-2,2) {$B$};
\node at (-6,2) {$B$};
\draw[very thick] (8,4) -- (0,4) ;
\draw[very thick] (4,0) -- (4,8) ;
\node at (2,6) {$C$};
\node at (6,6) {$C$};
\node at (2,2) {$C$};
\node at (6,2) {$C$};
\end{tikzpicture}
}\hspace{0.15cm}
\subfloat[Level 3]
{
\begin{tikzpicture}[scale=0.25]
\draw[very thick] (-8,-8) -- (-8,8) -- (8,8) -- (8,-8) -- cycle ;
\draw[very thick] (-8,0) -- (8,0) ;
\draw[very thick] (0,0) -- (0,-8) ;
\draw[very thick] (-4,0) -- (-4,-8) ;
\draw[very thick] (4,0) -- (4,-8) ;
\draw[very thick] (-8,-2) -- (8,-2) ;
\draw[very thick] (-8,-4) -- (8,-4) ;
\draw[very thick] (-8,-6) -- (8,-6) ;
\node at (6,-1) {$A$};
\node at (6,-3) {$A$};
\node at (-6,-1) {$A$};
\node at (-6,-3) {$A$};
\node at (6,-5) {$A$};
\node at (6,-7) {$A$};
\node at (-6,-5) {$A$};
\node at (-6,-7) {$A$};
\node at (2,-1) {$A$};
\node at (2,-3) {$A$};
\node at (-2,-1) {$A$};
\node at (-2,-3) {$A$};
\node at (2,-5) {$A$};
\node at (2,-7) {$A$};
\node at (-2,-5) {$A$};
\node at (-2,-7) {$A$};
\draw[very thick] (0,0) -- (0,8) ;
\draw[very thick] (-8,4) -- (0,4) ;
\draw[very thick] (-8,2) -- (0,2) ;
\draw[very thick] (-8,6) -- (0,6) ;
\draw[very thick] (-4,0) -- (-4,8) ;
\draw[very thick] (-4,0) -- (-4,8) ;
\draw[very thick] (-2,0) -- (-2,8) ;
\draw[very thick] (-6,0) -- (-6,8) ;
\node at (-1,7) {$B$};
\node at (-7,7) {$B$};
\node at (-1,1) {$B$};
\node at (-7,1) {$B$};
\node at (-3,7) {$B$};
\node at (-5,7) {$B$};
\node at (-3,1) {$B$};
\node at (-5,3) {$B$};
\node at (-1,5) {$B$};
\node at (-7,5) {$B$};
\node at (-1,3) {$B$};
\node at (-7,3) {$B$};
\node at (-3,5) {$B$};
\node at (-5,5) {$B$};
\node at (-3,3) {$B$};
\node at (-5,3) {$B$};
\node at (-5,1) {$B$};
\draw[very thick] (0,0) -- (0,8) ;
\draw[very thick] (8,4) -- (0,4) ;
\draw[very thick] (8,2) -- (0,2) ;
\draw[very thick] (8,6) -- (0,6) ;
\draw[very thick] (4,0) -- (4,8) ;
\draw[very thick] (2,0) -- (2,8) ;
\draw[very thick] (6,0) -- (6,8) ;
\node at (1,7) {$C$};
\node at (7,7) {$C$};
\node at (1,1) {$C$};
\node at (7,1) {$C$};
\node at (3,7) {$C$};
\node at (5,7) {$C$};
\node at (3,1) {$C$};
\node at (5,3) {$C$};
\node at (1,5) {$C$};
\node at (7,5) {$C$};
\node at (1,3) {$C$};
\node at (7,3) {$C$};
\node at (3,5) {$C$};
\node at (5,5) {$C$};
\node at (3,3) {$C$};
\node at (5,3) {$C$};
\node at (5,1) {$C$};
\end{tikzpicture}
}
\end{center}
\caption{Three levels of mesh refinement used to investigate the experimental order of convergence for the $h/p$ non-conforming DG approximation.}
\label{mesh3}
\end{figure}
The DG derivative matrix (i.e. the SBP operator) depends on the polynomial degree within each element. Therefore, in the case of $p$ refinement the SBP operator may differ between elements $A,B,C$.


We consider the DGSEM on Legendre-Gauss-Lobatto nodes as in \cite{Gassner2013}. To do so, we investigate the following configurations:
\begin{itemize}
\item Element $A$ with a degree $p_A=p$ operator in $x$- and $y$-direction\\
\item Element $B$ with a degree $p_B=p+1$ operator in $x$- and $y$-direction\\
\item Element $C$ with a degree $p_C=p$ operator in $x$- and $y$-direction,
\end{itemize}
with $p=2,3$.

With such element distributions we consider $p$ refinement along the line $x=5$ for $y\in[5,10]$ and $h$ refinement along the line $y=5$ for $x\in[0,10]$. To carefully treat the non-conforming interfaces we create the projection operators described in Appendix \ref{sec:App B}.
With these operators included in the non-conforming entropy stable scheme we obtain the experimental order of convergence (EOC) rates collected in Tables \ref{DGMix2} and \ref{DGMix3}.
\begin{table}[ht]
\begin{center}
\textbf{DG operators with mixed polynomial degree}\\
\vspace{0.3cm}
\begin{minipage}{0.4\textwidth}
\begin{center}
\begin{tabular}{c|c|c}
\hline
DOFS & $L_2$ &  EOC\\
\hline
544   &1.90E-01&\\
2176  &3.06E-02&2.6\\
8704  &4.28E-03&2.8\\
34826 &8.44E-04&2.3\\
139264&1.80E-04&2.2\\
\hline
\end{tabular}\\[0.1cm]
\caption{Experimental order of convergence for the non-conforming entropy stable scheme using DG-operators of degree two and three.}
\label{DGMix2}
\end{center}
\end{minipage}
\qquad\qquad
\begin{minipage}{0.4\textwidth}
\begin{center}
\begin{tabular}{c|c|c}
\hline
DOFS & $L_2$ &  EOC\\
\hline
912   &2.55E-02&\\
3648  &2.02E-03&3.7\\
14592 &1.81E-04&3.5\\
58368 &1.98E-05&3.2\\
233472&2.28E-06&3.1\\
\hline
\end{tabular}\\[0.1cm]
\caption{Experimental order of convergence for the non-conforming entropy stable scheme using DG-operators of degree three and four.}
\label{DGMix3}
\end{center}
\end{minipage}
\end{center}
\end{table}

We verify a convergence order slightly higher than $p$, where $p=\min\{p_A,p_B,p_C\}$. This result is also documented for non-conforming schemes as in \cite{Friedrich2016} for linear problems. In comparison, conforming schemes have an EOC of $p+1$. The order reduction occurs presumably because of the degree of the projection operators.

Focusing on two elements with SBP operators $(\mass_A, \D_A)$ and $(\mass_B, \D_B)$, where $\D_A$ and $\D_B$ are of degree $p_A$ and $p_B$. For SBP operators constructed on LGL nodes (DGSEM \cite{Gassner2013}) or on uniform distributed nodes (SBP-SAT finite difference \cite{DCDRF2014}) the norm matrices $\mass_A$ and $\mass_B$ can integrate polynomials of degree $2p_A-1$ and $2p_B-1$ exactly. Let $\mat{P}_{A2B}$ denote projection operator of degree $p_1$ and $\mat{P}_{B2A}$ denote the projection operator of degree $p_2$, then Lundquist and Nordstr\"om \cite{Nordstrom2015} proved that
\begin{equation}
p_1+p_2\le 2p_{min}-1,
\end{equation}
where $p_{min}=\min\{p_A,p_B\}$. So, when considering non-conforming schemes, not all projection operators can be of degree $p_{min}$. The upper bound of $2p_{min}-1$ is due to the accuracy of the integration matrix. For this reason Friedrich et al. \cite{Friedrich2016} created a special set of SBP-finite difference operators, where the norm matrix can integrate polynomials of degree $>2p_{min}$ exactly. With these operators it is possible to construct projection operators of the same degree as the SBP-operators (degree preservation). The construction of the projection operators is outlined in \cite{Friedrich2016}. Convergence test with these operators are documented in Appendix \ref{sec:App C} and show a full convergence order of $p+1$ in the non-conforming case.

To summarize, the non-conforming entropy stable scheme has the flexibility to chose different nodal distribution aswell as elements of different sizes and obtains an experimental order of convergence of $p$.


\subsection{Verification of Primary and Entropy Conservation/Stability}
In this section we numerically verify primary conservation and entropy conservation/stability for the new derived scheme. We first demonstrate entropy conservation which was the result of Theorem \ref{thm: p} and Corollary \ref{cor}. Therefore we consider the entropy conservative flux of Ismail and Roe \cite{Ismail2009} without dissipation. To verify the conservation of entropy, we consider the mesh in Fig. \ref{mesh3}(c) on $\Omega=[0,1]\times[0,1]$ with periodic boundary conditions. For each type of element we consider a DG operators with $p_A=p_C=3$ and $p_B=4$. To calculate the discrete growth in the primary quantities and entropy we rewrite \eqref{eq:splitDG} by
\begin{equation}\label{eq:splitDGRes}
J\left(\bm{U}_t\right)_{ij}+\bm{Res}\left(\bm{U}_t\right)_{ij} =0,
\end{equation}
where
\begin{equation}
\resizebox{\hsize}{!}{$
\begin{aligned}
\bm{Res}\left(\bm{U}_t\right)_{ij} =&+ \frac{1}{\mass_{ii}}\left(\delta_{iN}\left[\tilde{\bm{F}}^{\EC}(1,\eta_j;\hat{n}) - \tilde{\bm{F}}_{Nj}\right] - \delta_{i0}\left[\tilde{\bm{F}}^{\EC}(-1,\eta_j;\hat{n}) - \tilde{\bm{F}}_{0j}\right]\right)+2\sum_{m=0}^N \D_{im}\tilde{\bm{F}}^{\EC}\left(\bm{U}_{ij},\bm{U}_{mj}\right)\\
&+ \frac{1}{\mass_{jj}}\left(\delta_{jN}\left[\tilde{\bm{G}}^{\EC}(\xi_i,1;\hat{n}) - \tilde{\bm{G}}_{iN}\right] - \delta_{j0}\left[\tilde{\bm{G}}^{\EC}(\xi_i,-1;\hat{n}) - \tilde{\bm{G}}_{i0}\right]\right)+2\sum_{m=0}^N \D_{jm}\tilde{\bm{G}}^{\EC}\left(\bm{U}_{ij},\bm{U}_{im}\right).\\
\end{aligned}$}
\end{equation}
The growth in entropy is computed by contracting \eqref{eq:splitDGRes} with the vector of entropy variables, i.e.,
\begin{equation}
\label{eq:splitDGRes2}
J\bm{V}_{ij}^T\left(\bm{U}_t\right)_{ij}=-\bm{V}_{ij}^T\bm{Res}\left(\bm{U}_t\right)_{ij}
\Leftrightarrow J\dEntij=-\bm{V}_{ij}^T\bm{Res}\left(\bm{U}_t\right)_{ij},
\end{equation}
where we use the definition of the entropy variables \eqref{eq:entVars} to obtain the temporal derivative, $\dEntij$, at each LGL node. As shown in Theorem \ref{thm: diss}, the scheme is primary and entropy conservative when no interface dissipation is included, meaning that
\begin{equation}
\begin{split}
\sum\limits_{\mathrm{all\,elements}} J\! \sum_{i,j=0}^N\omega_i\omega_j \left(\bm{U}_t\right)_{ij}= 0,\\
\sum\limits_{\mathrm{all\,elements}} J\! \sum_{i,j=0}^N\omega_i\omega_j \dEntij= 0,
\end{split}
\end{equation}
for all time. We verify this result numerically inserting \eqref{eq:splitDGRes2} and calculate
\begin{equation}
\begin{split}
\bm{\dTotCon}&:=-\sum\limits_{\mathrm{all\,elements}}~\sum_{i,j=0}^N\omega_i\omega_j \bm{Res}\left(\bm{U}_t\right)_{ij}= 0,\\
\dTotEnt&:=-\sum\limits_{\mathrm{all\,elements}}~\sum_{i,j=0}^N\omega_i\omega_j\bm{V}_{ij}^T\bm{Res}\left(\bm{U}_t\right)_{ij},
\end{split}
\end{equation}
using a discontinuous initial condition
\begin{equation}
\begin{pmatrix}
\rho\\
u\\
v\\
p\\
\end{pmatrix}
=
\begin{pmatrix}
\mu_{1,1}\\
\mu_{1,2}\\
\mu_{1,3}\\
\mu_{1,4}\\
\end{pmatrix}
\quad\text{ if } x\le y,\qquad
\begin{pmatrix}
\rho\\
u\\
v\\
p\\
\end{pmatrix}
=
\begin{pmatrix}
\mu_{2,1}\\
\mu_{2,2}\\
\mu_{2,3}\\
\mu_{2,4}\\
\end{pmatrix}
\quad\text{ if } x> y.
\end{equation}
Here $\mu_{k,l}$ are uniformly generated random numbers in $[0,1]$. The random initial condition is chosen to demonstrate entropy conservation independent of the initial condition. We calculate $\bm{\dTotCon}$ and $\dEnt$ for $1000$ different initial conditions which gives us $(\dTotCon)_{lk}$ and $(\dTotEnt)_{k}$ for $k=1,\dots,1000$ and $l=1,\dots,4$. Within the $L_2$ product we obtain
\begin{table}[ht]
\begin{center}
\textbf{Verification of primary and entropy conservation}\\
\vspace{0.3cm}
\begin{tabular}{c|c|c|c|c}
\hline
$L_2(\bm{\dTotEnt})$ & $L_2((\bm{\dTotCon})_{1,:})$ & $L_2((\bm{\dTotCon)_{2,:}})$ & $L_2((\bm{\dTotCon)_{4,:}})$ & $L_2((\bm{\dTotCon)_{4,:}})$\\
\hline
4.56E-14 & 2.57E-14 & 1.35E-14 & 2.26E-14 & 8.53E-14\\
\hline
\end{tabular}\\[0.1cm]
\caption{Calculating the growth of the primary quantities $\bm{\dTotCon}$ and entropy $\bm{\dTotCon}$ for 1000 different random initial conditions with the new scheme. The growth is presented within the $L_2$ product. All values are near machine precision which demonstrates primary and entropy conservation.}
\label{TabEnt1}
\end{center}
\end{table}
 
In Table \ref{TabEnt1} we verify primary and entropy conservation. In comparison, when considering the same setup and calculating the numerical flux with the standard mortar method by Kopriva \cite{Kopriva1996b} we verify primary conservation but the method \textit{is not} entropy conservative, see Table \ref{TabEnt2}. 

\begin{table}[ht]
\begin{center}
\textbf{Calculating the growth in the primary quantities and entropy with the standard mortar method}\\
\vspace{0.3cm}
\begin{tabular}{c|c|c|c|c}
\hline
$L_2(\bm{\dTotEnt})$ & $L_2((\bm{\dTotCon})_{1,:})$ & $L_2((\bm{\dTotCon)_{2,:}})$ & $L_2((\bm{\dTotCon)_{4,:}})$ & $L_2((\bm{\dTotCon)_{4,:}})$\\
\hline
3.07E-02 & 5.90E-15 & 2.62E-14 & 6.27E-15 & 1.04E-13\\
\hline
\end{tabular}\\[0.1cm]
\caption{Calculating the growth of the primary quantities $\bm{\dTotCon}$ and entropy $\bm{\dTotCon}$ for 1000 different random initial conditions with the mortar element method. The growth is presented within the $L_2$- product. Here, we verify conservation of the primary quantities but not entropy conservation.}
\label{TabEnt2}
\end{center}
\end{table}
Next, we demonstrate the increased robustness of the novel entropy conservative, non-conforming scheme. Therefore, we approximate the total entropy in time by
\begin{equation}
\TotEnt:=\sum\limits_{\mathrm{all\,elements}} J\! \sum_{i,j=0}^N\omega_i\omega_j \left(S\right)_{ij}
\end{equation}
over the time domain $t\in [0,T]$, where we choose $T=25$ and $CFL=0.5$. For the Euler equations the entropy function is defined by
\begin{equation}
\Ent=-\frac{\rho}{\gamma-1}\log\left(\frac{p}{\rho^{\gamma}}\right).
\end{equation}
We solve for the total entropy in time with the low-storage Runge-Kutta time integration method of Carpenter and Kennedy \cite{Kennedy1994} using a discontinuous initial condition
\begin{equation}
\begin{pmatrix}
\rho\\
u\\
v\\
p\\
\end{pmatrix}
=
\begin{pmatrix}
1.08\\
0.2\\
0.01\\
0.95\\
\end{pmatrix}
\quad\text{ if } x\le y,\qquad
\begin{pmatrix}
\rho\\
u\\
v\\
p\\
\end{pmatrix}
=
\begin{pmatrix}
1\\
10^{-12}\\
10^{-12}\\
1\\
\end{pmatrix}
\quad\text{ if } x> y,
\end{equation}
and periodic boundaries. Again, we use the new derived method and the classical mortar method \cite{Tan2012,Kopriva1996b}. In Fig. \ref{fig:Mortar} we plot the temporal evolution of the entropy for the standard mortar method against the newly derived scheme.

\begin{figure}[ht]
\begin{center}
	  \includegraphics[scale=0.35]{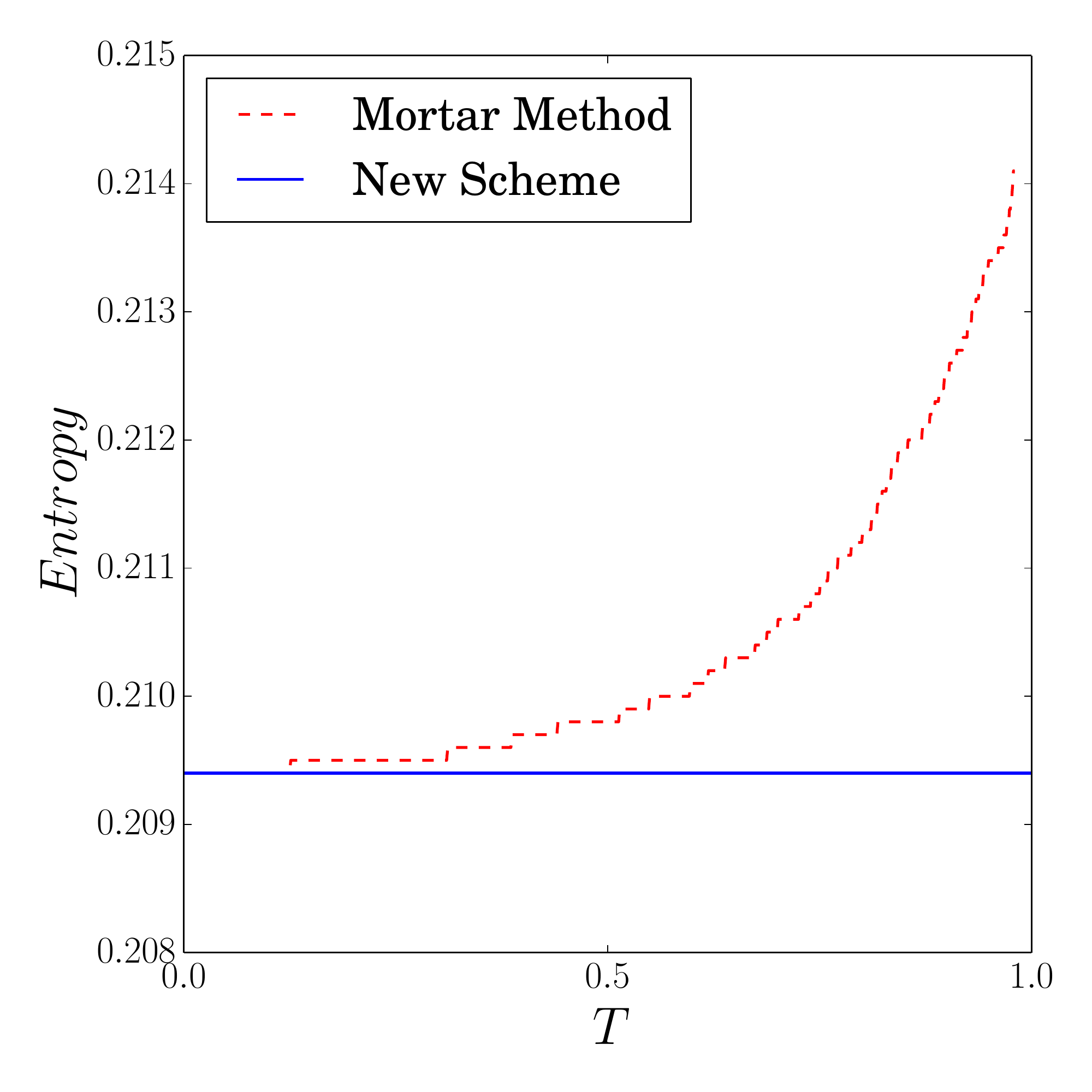}
		\caption{Evolution of the total entropy comparing the behavior of the standard mortar method against the new scheme derived in this work with an entropy conservative surface flux. We see that the total entropy grows for the mortar method with this test configuration whereas the entropy is conserved for the new scheme. In fact, the mortar method crashes at $t\approx 1$.}
		\label{fig:Mortar}
	\end{center}
\end{figure}

The new scheme conserves the total entropy. However, for the mortar method we observe an unpredictable behavior of the entropy for $t<1$ and note that at $t\approx 1$ the approach even crashes. This has been verified for the CFL numbers $CFL=0.5;~0.25;~0.125;~0.0625$ and demonstrates the enhanced robustness of entropy conserving/stable schemes.

Finally, we verify the entropy stability and conservation of the primary quantities. Therefore, we include dissipation in a local Lax-Friedrichs sense as described in Sec. \ref{sec:Dissipation}. For this test we use the same configuration as for verifying entropy conservation and set $CFL=0.5$.
\begin{figure}[ht]
\begin{center}
	\begin{minipage}{0.4\textwidth}
\begin{flushleft}
		\vspace{0.785cm}
	  \includegraphics[width=\textwidth]{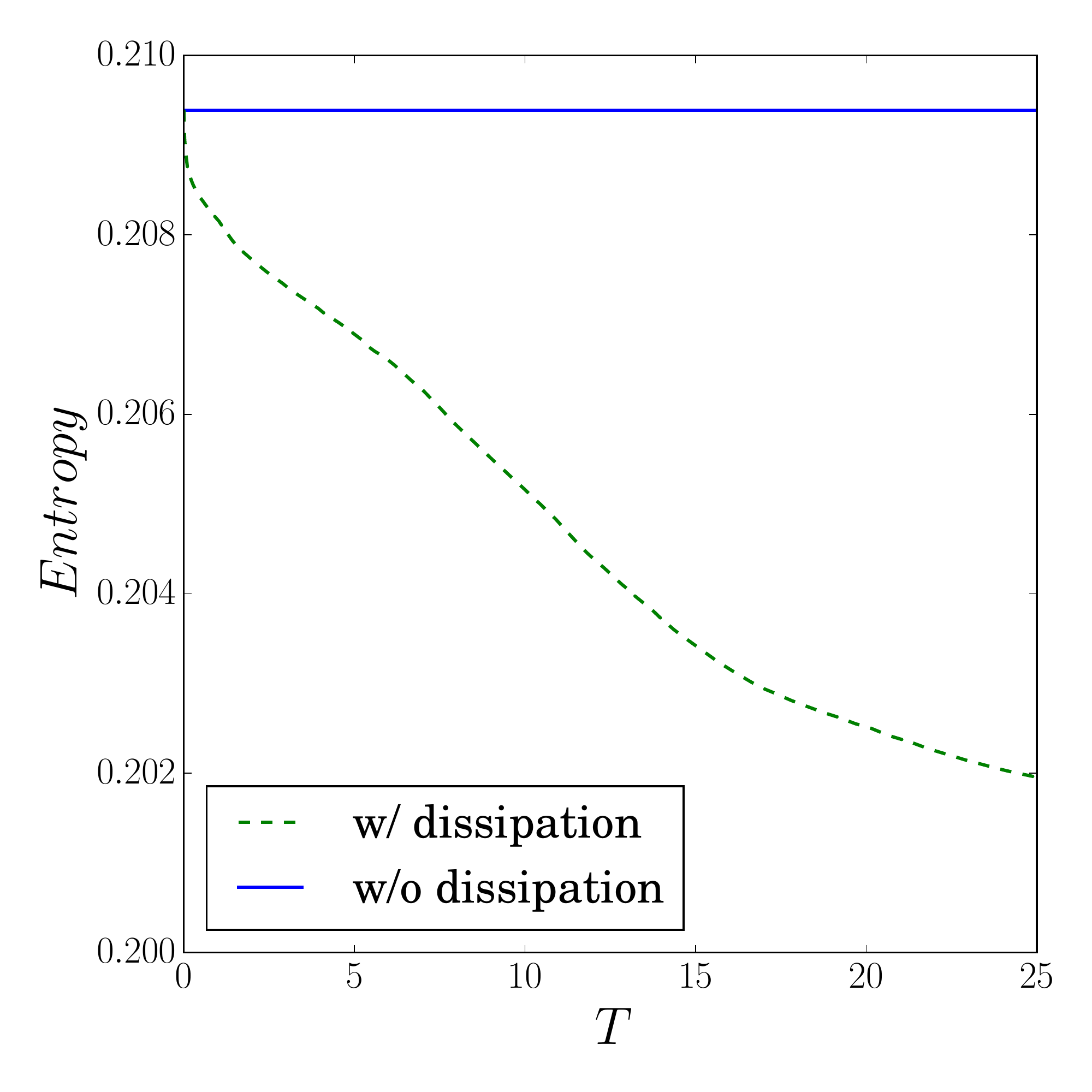}
		\caption{Evolution of the total entropy of the solution with and without dissipation. We see that the total entropy is conserved when no interface dissipation is included and total energy decays with interface dissipation.}
		\label{fig: Ent}
\end{flushleft}
	\end{minipage}
	\qquad
	\begin{minipage}{0.4\textwidth}
	  \includegraphics[width=\textwidth]{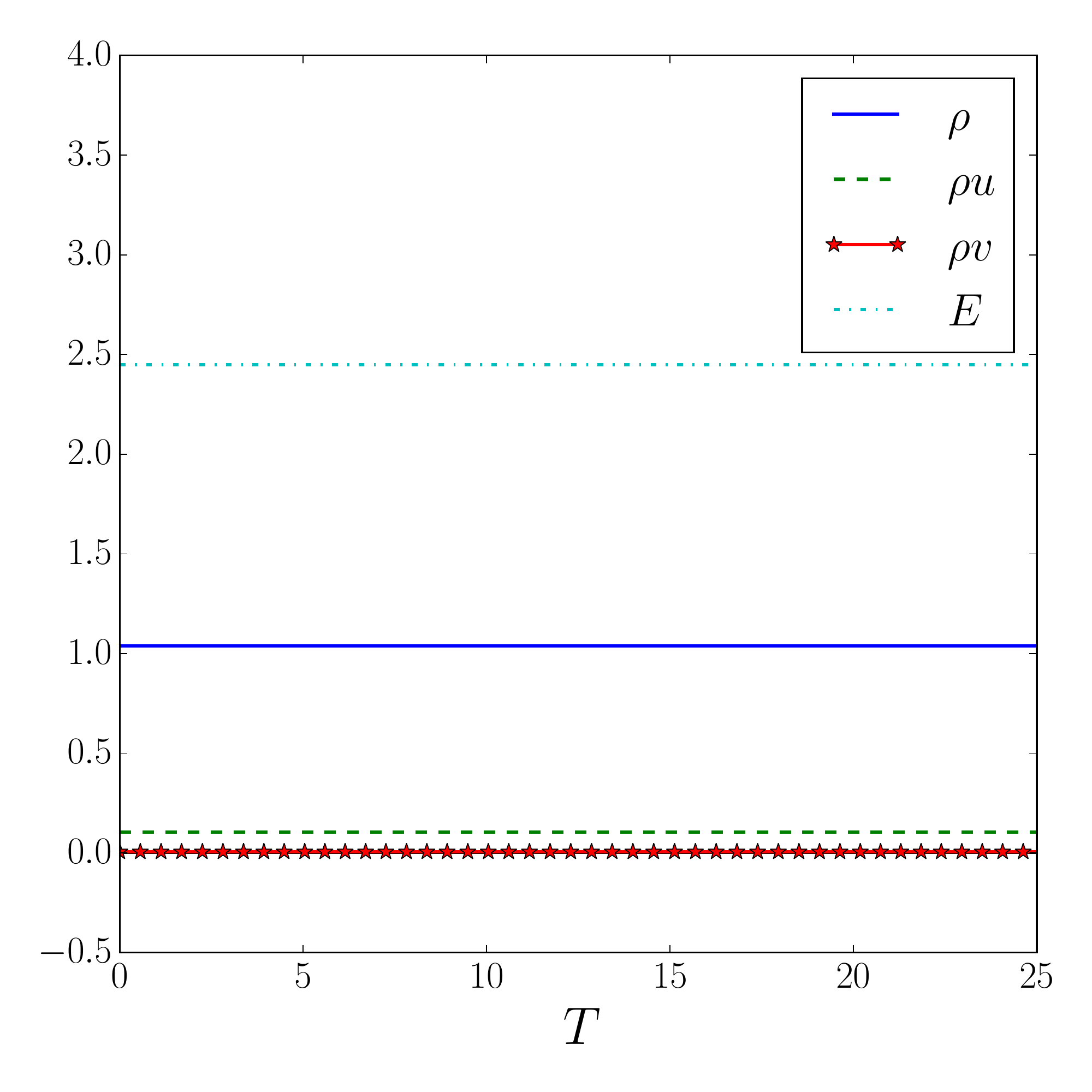}
		\caption{A plot that demonstrates the conservation of the primary quantities. The plots do not depend on interface dissipation.}
		\label{fig: Con}
	\end{minipage}
	\end{center}
\end{figure}

In Fig. \ref{fig: Con} we can see that the primary quantities are conserved over time. Note, that in comparison the non-entropy conserving mortar scheme crashes at $t\approx 1$. Also, we note that the plot remains the same whether or not dissipation is included. In Fig. \ref{fig: Ent} we can see that the total entropy remains constant when considering an entropy conservative flux. Therefore, when including dissipation, the total entropy decays which numerically verifies entropy stability.

\section{Conclusion}

In this work we derived a $h/p$ non-conforming primary conservative and entropy stable discontinuous Galerkin spectral element approximation with the summation-by-parts (SBP) property for non-linear conservation laws. We first examined the standard mortar method and found that it did not guarantee entropy conservation/stability for non-linear problems. Hence, we present a modification of the mortar method with special attention given to the projection operators between non-conforming elements. As an extension of the work \cite{Carpenter2016} we extend an entropy stable $p$ non-conforming discretization to a more general $h/p$ non-conforming setup. Neither the nodes nor the interface of two neighboring elements need to coincide in the novel approach. Throughout the derivations in this paper it was required to consider SBP operators, like that for the LGL nodal discontinuous Galerkin spectral method, as these operators mimic the integration-by-parts rule in a discrete manner. To demonstrate the high-order accuracy and entropy conservation/stability of the non-conforming DGSEM we selected the two-dimensional Euler equations. However, we reiterate that the proofs contained herein are general for systems of non-linear hyperbolic conservation laws and directly apply to all diagonal norm SBP operators, as e.g. presented in Appendix \ref{sec:App C}.


\acknowledgement{Lucas Friedrich and Andrew Winters were funded by the Deutsche Forschungsgemeinschaft (DFG) grant TA 2160/1-1. Special thanks goes to the Albertus Magnus Graduate Center (AMGC) of the University of Cologne for funding Lucas Friedrich's visit to the National Institute of Aerospace, Hampton, VA, USA. Gregor Gassner has been supported by the European Research Council (ERC) under the European Union's Eights Framework Program Horizon 2020 with the research project \textit{Extreme}, ERC grant agreement no. 714487. This work was partially performed on the Cologne High Efficiency Operating Platform for Sciences (CHEOPS) at the Regionales Rechenzentrum K\"{o}ln (RRZK) at the University of Cologne.}

\bibliographystyle{plain}
\bibliography{Bib}

\appendix

\section{Derivations of the Growth of Primary Quantities and Entropy}\label{sec:App A}
The proof below is the same result as presented by Fisher et al. \cite{Fisher2013}. For completeness, we re-derive the proof in our notation. We analyze the two dimensional discretizaion of \eqref{eq:ECDG} on a single element.
\begin{equation}
J\omega_i\omega_j\left(\bm{U}_{t}\right)_{ij}+\omega_j\mathcal{L}(\bm{U}_{ij})_x+\omega_i\mathcal{L}(\bm{U}_{ij})_y=0,
\end{equation}
with
\begin{equation}
\begin{aligned}
\mathcal{L}(\bm{U}_{ij})_x&=2\sum_{m=0}^N\omega_i\D_{im}\fsharp(\ukij,\ukmj)-\left(\delta_{iN}[\tilde{\bm{F}}-\fstar]_{Nj}-\delta_{i0}[\tilde{\bm{F}}-\fstar]_{0j}\right),\\
\mathcal{L}(\bm{U}_{ij})_y&=2\sum_{m=0}^N\omega_j\D_{jm}\gsharp(\ukij,\ukim)-\left(\delta_{Nj}[\tilde{\bm{G}}-\gstar]_{iN}-\delta_{0j}[\tilde{\bm{G}}-\gstar]_{i0}\right).
\end{aligned}
\end{equation}
Assuming, that $\fsharp$ and $\gsharp$ satisfy the appropriate entropy condition \eqref{eq:entCondition}
\begin{equation}
\begin{split}
\fsharp(\ukij,\uk_{ml})^T\left(\bm{V}_{ij}-\bm{V}_{ml}^{\eqInd}\right)=\tilde\Psi^f_{ij}-\tilde\Psi^f_{ml},\\
\gsharp(\ukij,\uk_{ml})^T\left(\bm{V}_{ij}-\bm{V}_{ml}^{\eqInd}\right)=\tilde\Psi^g_{ij}-\tilde\Psi^g_{ml}.
\end{split}
\end{equation}
First, we derive the growth of the primary quantities on each element $E_k, k=1,\ldots,K$. Summing over all nodes $i,j=0,\dots,N$ yields
\begin{equation}
\underbrace{J\sum_{i,j=0}^N\omega_i\omega_j\left(\bm{U}_{t}\right)_{ij}}_{\approx\int \bm{U}_t \,\mathrm{d}E}+\sum_{j=0}^N\omega_j\sum_{i=0}^N\mathcal{L}(\bm{U}_{ij})_x+\sum_{i=0}^N\omega_i\sum_{j=0}^N\mathcal{L}(\bm{U}_{ij})_y=0,
\end{equation}
with
\begin{equation}
\sum_{i=0}^N\mathcal{L}(\bm{U}_{ij})_x=2\sum_{i=0}^N\sum_{m=0}^N\mat{Q}_{im}\fsharp(\ukij,\ukmj)-[\tilde{\bm{F}}-\fstar]_{Nj}+[\tilde{\bm{F}}-\fstar]_{0j}.
\end{equation}
Using the SBP property of the matrices $2\mat Q=\mat Q-\mat Q^T+\mat B$ we find
\begin{equation}
\sum_{i=0}^N\mathcal{L}(\bm{U}_{ij})_x=\sum_{i=0}^N\sum_{m=0}^N \mat{Q}_{im}\fsharp(\ukij,\ukmj)-\sum_{i=0}^N\sum_{m=0}^N\mat{Q}_{im}\fsharp(\ukmj,\ukij)+\fstar_{Nj}-\fstar_{0j}.
\end{equation}
Due to nearly skew-symmetric nature of $\mat{Q}$ we arrive at
\begin{equation}
\sum_{i=0}^N\mathcal{L}(\bm{U}_{ij})_x=\fstar_{Nj}-\fstar_{0j}.
\end{equation}
And similar for $\sum\limits_{j=0}^N\mathcal{L}(\bm{U}_{ij})_y$ we have
\begin{equation}
\sum_{j=0}^N\mathcal{L}(\bm{U}_{ij})_x=\gstar_{iN}-\gstar_{i0}.
\end{equation}
Together both directions yield
\begin{equation}
\underbrace{J\sum_{i,j=0}^N\omega_i\omega_j\left(\bm{U}_{t}\right)_{ij}}_{\approx\int \bm{U}_t \,\mathrm{d}E}=-\sum_{j=0}^N\omega_j\left(\fstar_{Nj}-\fstar_{0j}\right)-\sum_{i=0}^N\omega_i\left(\gstar_{iN}-\gstar_{i0}\right),
\end{equation}
which is precisely \eqref{ThmdU}.

Next, we derive the entropy growth on the single element. To do so, we pre-multiply with the entropy variables and sum over all nodes to get
\begin{equation}
J\sum_{i,j=0}^N\omega_i\omega_j\underbrace{\bm{V}_{ij}^T\left(\bm{U}_{t}\right)_{ij}}_{=:\dEntij}+\sum_{j=0}^N\omega_j\sum_{i=0}^N\bm{V}_{ij}^T\mathcal{L}(\bm{U}_{ij})_x + \sum_{i=0}^N\omega_j\sum_{j=0}^N\bm{V}_{ij}^T\mathcal{L}(\bm{U}_{ij})_y=0,
\end{equation}
with
\begin{equation}
\begin{aligned}
\sum_{i=0}^N\bm{V}_{ij}^T\mathcal{L}(\bm{U}_{ij})_x=&2\sum_{i=0}^N\sum_{m=0}^N\mat{Q}_{im}\bm{V}_{ij}^T\fsharp(\ukij,\ukmj)-\left[\bm{V}^T\tilde{\bm{F}}-\bm{V}^T\fstar\right]_{Nj}+\left[\bm{V}^T\tilde{\bm{F}}+\bm{V}^T\fstar\right]_{0j}.
\end{aligned}
\end{equation}
Again using $2\mat Q=\mat Q-\mat Q^T+\mat B$ we have
\begin{equation}
\begin{aligned}
\sum_{i=0}^N\bm{V}_{ij}^T\mathcal{L}(\bm{U}_{ij})_x=&\sum_{i=0}^N\sum_{m=0}^N\mat{Q}_{im}\bm{V}_{ij}^T\fsharp(\ukij,\ukmj)-\sum_{i=0}^N\sum_{m=0}^N\mat{Q}_{im}\bm{V}_{ji}^T\fsharp(\ukmj,\ukij)\\
&+\bm{V}_{Nj}^T\fstar_{Nj}-\bm{V}^T_{0j}\fstar_{0j}.
\end{aligned}
\end{equation}
Due to entropy conservation condition \eqref{eq:entCondition} and the consistency of the derivative matrix, i.e. $\D\One=\bm 0~\left(\Leftrightarrow\mat Q\One=\bm 0\right)$ we find
\begin{equation}
\begin{aligned}
\sum_{i=0}^N\bm{V}_{ij}^T\mathcal{L}(\bm{U}_{ij})_x
=&\sum_{i=0}^N\sum_{m=0}^N\mat{Q}_{im}\underbrace{\left(\left(\fsharp(\ukij,\ukmj)\right)^T\left(\bm{V}_{ij}-\bm{V}_{ji}\right)\right)}_{=\tilde{\Psi}^f_{ij}-\tilde{\Psi}^f_{mj}}+\bm{V}_{Nj}^T\fstar_{Nj}-\bm{V}_{0j}^T\fstar_{0j},\\
=&\sum_{i=0}^N\tilde{\Psi}^f_{ij}\underbrace{\sum_{m=0}^N\mat{Q}_{im}}_{=0}-\sum_{i=0}^N\sum_{m=0}^N\underbrace{\mat{Q}_{im}}_{\mat{B}_{im}-\mat{Q}_{mi}}\tilde{\Psi}^f_{mj}+\bm{V}_{Nj}^T\fstar_{Nj}-\bm{V}_{0j}^T\fstar_{0j},\\
=&-\sum_{i=0}^N\sum_{m=0}^N\mat{B}_{im}\tilde{\Psi}^f_{mj}+\sum_{m=0}^N\tilde{\Psi}^f_{mj}\underbrace{\sum_{i=0}^N\mat{Q}_{mi}}_{=0}+\bm{V}_{Nj}^T\fstar_{Nj}-\bm{V}_{0j}^T\fstar_{0j},\\
=&\left(\bm{V}_{Nj}^T\fstar_{Nj}-\tilde{\Psi}^f_{Nj}\right)-\left(\bm{V}_{0j}^T\fstar_{0j}-\tilde{\Psi}^f_{0j}\right).
\end{aligned}
\end{equation}
And similar for $\sum\limits_{j=0}^N\bm{V}_{ij}^T\mathcal{L}(\bm{U}_{ij})_y$ we get
\begin{equation}
\sum_{j=0}^N\bm{V}_{ij}^T\mathcal{L}(\bm{U}_{ij})_y=\left(\bm{V}_{iN}^T\gstar_{iN}-\tilde{\Psi}^g_{iN}\right)-\left(\bm{V}_{i0}^T\gstar_{i0}-\tilde{\Psi}^g_{i0}\right)
\end{equation}
Both directions together yield
\begin{equation}
\begin{split}
\underbrace{J\sum_{i,j=0}^N\omega_i\omega_j\dEntij}_{\approx\int \dEnt \,\mathrm{d}E}=&-\sum_{j=0}^N\omega_j\left(\left(\bm{V}_{Nj}^T\fstar_{Nj}-\tilde{\Psi}^f_{Nj}\right)-\left(\bm{V}_{0j}^k\fstar_{0j}-\tilde{\Psi}^f_{0j}\right)\right)\\
&-\sum_{i=0}^N\omega_i\left(\left(\bm{V}_{iN}^T\gstar_{iN}-\tilde{\Psi}^g_{iN}\right)-\left(\bm{V}_{i0}^T\gstar_{i0}-\tilde{\Psi}^g_{i0}\right)\right),
\end{split}
\end{equation}
which is precisely \eqref{ThmdEnt}.

\section{Projection operators for Discontinuous Galerkin methods}\label{sec:App B}

The projection operators for DG methods are constructed with the \textit{Mortar Element Method} by Kopriva \cite{Kopriva2002}.

Here we assume two neighboring elements with a single coinciding interface as in Fig. \ref{Confmesh}. Let $N,M$ denote the polynomial degrees of both elements with corresponding one-dimensional nodes $x_0^N,\dots,x_N^N$ and $x_0^M,\dots,x_M^M$ and integration weights $\omega_0^N,\dots,\omega_N^N$ and $\omega_0^M,\dots,\omega_M^M$. The corresponding norm matrices are defined as $\mass_N=\diag(\omega_0^N,\dots,\omega_N^N)$ and $\mass_M=\diag(\omega_0^M,\dots,\omega_M^M)$ and each element is equipped with a set of Lagrange basis functions $l_0^N,\dots,l_N^N$ and $l_0^M,\dots,l_M^M$.

As for the elements, the mortar also consists of a set of nodes, integration weights, norm matrix and Lagrange basis functions. Without lose of generality, we assume $N<M$. Therefore, the polynomial order on the mortar $N_{\Xi} = \max\{N,M\} = M$. So the mortar will simply copy the solution data form the element with the higher polynomial degree $M$ because the nodal distributions are identical. Thus, the projection operator from the element to the mortar as well as back from the mortar to the element are simply the identity matrix of size $M$, i.e. $\mat P_{M2\Xi}=\mat I^M$ and $\mat P_{\Xi2M}=\mat I^M$. Next, we briefly describe how to project the element of degree $N$ to the mortar and back.

\textbf{Step 1 (Projection from element of degree $N$ to the mortar):} Assume we have a discrete evaluated function $\bm f=(f_0,\dots,f_N)^T$ with $f(x)=\sum_{i=0}^N\ell_i^N(x)f_i$. We want to project this function onto the mortar to obtain $\bm{f}^{\Xi}=(f^{\Xi}_0,\dots, f^{\Xi}_{M})^T$ with $ f^{\Xi}(x)=\sum_{j=0}^{M}\ell_j^M(x) f^{\Xi}_j$. Note, that ${f}(x)\neq {f}^\Xi(x)$ for a polynomial of higher degree. In \cite{Kopriva2002} the operator $P_{N2\Xi}$ is created by a $L_2$ projection on the mortar
\begin{equation}\label{App projection}
\ip{f}{L_2}{\ell_j^M} = \ip{f^{\Xi}}{L_2}{\ell_j^M}
\Leftrightarrow
\sum_{i=0}^N\ip{\ell_i^N}{L_2}{\ell_j^M}f_i=\sum_{i=0}^M\ip{\ell_i^M}{L_2}{\ell_j^M}\ f^\Xi_i,
\end{equation}
for $j=0,\dots,M$. Here, the $L_2$ inner products are evaluated discretely using the appropriate norm matrices. The $L_2$ inner product on the left in \eqref{App projection} is evaluated exactly due to the high-order nature of the LGL quadrature and the assumption that $N<M$. Therefore, using $M$-LGL nodes and weights we have
\begin{equation}\label{App left}
\ip{\ell_i^N}{L_2}{\ell_j^M} = \ip{\ell_i^N}{M}{\ell_j^M} = \sum_{k=0}^M\omega_k^M\ell_i^N(x_k^M)\ell_j^M(x_k^M)=\sum_{k=0}^M\omega_k^M\ell_i^N(x_k^M)\delta_{jk}=\omega_j^M\ell_i^N(x_j^M),
\end{equation}
for $i=0,\ldots,N$, $j=0,\ldots,M$ and use the Kronecker delta property of the Lagrange basis. On the right side of \eqref{App projection} we evaluate an inner product of two polynomial basis functions of order $M$. Therefore, due to the exactness of the LGL quadrature, the $L_2$ inner product is approximated by an integration rule with mass lumping, e.g. \cite{Tan2012},
\begin{equation}\label{App right}
\ip{\ell_i^M}{L_2}{\ell_j^M} \approx \ip{\ell_i^M}{M}{\ell_j^M} = \sum_{k=0}^M\omega_k^M\ell_i^M(x_k^M)\ell_j^M(x_k^M)=\sum_{k=0}^M\omega_k^M\delta_{ik} \delta_{jk} = \delta_{ij}\omega_j^M,
\end{equation}
for $i,j=0,\ldots,M$. Next, we define \textit{interpolation} operators
\begin{equation}
[\mat L_{N2\Xi}]_{ij}:= \ell_j^N(x_i^M),
\end{equation}
with $i=0,\dots,N$ and $j=0,\dots,M$ to rewrite \eqref{App projection} in a compact matrix-vector notation
\begin{equation}
\mass_M\mat L_{N2\Xi} \bm f=\mass_M\bm{f}^\Xi
\Leftrightarrow \underbrace{\mat L_{N2\Xi}}_{:= \mat{P}_{N2\Xi}}\bm f=\bm{f}^\Xi.
\end{equation}
So the projection operator to move the solution from the element with $N$ nodes onto the mortar is equivalent to an interpolation operator. However, this does not hold for projecting the solution from the mortar back to the element.

\textbf{Step 2 (Projection from the mortar to element of degree $N$):} To construct the operator $\mat{P}_{\Xi 2N}$ we consider the $L_2$ projection from the mortar back to an element with $N$ nodes. Here, we assume a discrete evaluation of the solution on the mortar $\bm f^\Xi=(f_0^\Xi,\dots,f_M^\Xi)^T$ with $f^\Xi(x)=\sum_{i=0}^M \ell_i^M(x)f_i^\Xi$ and seek the solution on the element $\bm{f}=( f_0,\dots, f_N)^T$ with $ f(x)=\sum_{i=0}^N \ell_i^N(x) f_i$. The $L_2$ projection back to the element is
\begin{equation}\label{AppBackprojection}
\ip{f^{\Xi}}{L_2}{\ell_j^N}=\ip{f}{L_2}{\ell_j^N}
\Leftrightarrow\sum_{i=0}^M \ip{\ell_i^M}{L_2}{\ell_j^N}f_i^\Xi=\sum_{i=0}^N\ip{\ell_i^N}{L_2}{\ell_j^N} f_i,
\end{equation}
for $j=0,\dots,N$. The $L_2$ inner product on the left in \eqref{AppBackprojection} is computed exactly using $M$-LGL points and the $L_2$ inner product on the right in \eqref{AppBackprojection} is approximated with mass lumping at $N$-LGL nodes. Thus, we obtain
\begin{equation}
\ip{\ell_i^M}{L_2}{\ell_j^N}=\ip{\ell_i^M}{M}{\ell_j^N}=\sum_{k=0}^M\omega_k^M \ell_i^M(x_k^M) \ell_j^N(x_k^M)=\omega_i^M \ell_j^N(x_i^M),
\end{equation}
where $i=0,\ldots,M$, $j = 0,\ldots,N$ and
\begin{equation}
\ip{\ell_i^N}{L_2}{\ell_j^N}\approx\ip{\ell_i^N}{N}{\ell_j^N}=\sum_{k=0}^N\omega_i^N \ell_i^N(x_k^N) \ell_j^N(x_k^N)=\delta_{ij}\omega_j^N,
\end{equation}
for $i,j=0,\ldots,N$. Again, we write \eqref{AppBackprojection} in a compact matrix-vector notation which gives us
\begin{equation}
\mat L_{N2\Xi}^T\mass_M\bm f^\Xi=\mass_N\bm{f}.
\end{equation}
As $\mat L_{N2\Xi}=\mat P_{N2\Xi}$ we obtain
\begin{equation}
\underbrace{\mass_N^{-1}\mat P_{N2\Xi}^T\mass_M}_{:=\mat P_{\Xi2N}}\bm f^\Xi=\bm{f},
\end{equation}
where we introduce the projection operator (not interpolation operator) from the mortar back to the element with $N$ nodes. With this approach we constructed projection operators satisfying the $\mass$-compatibility condition \eqref{7}, i.e.,
\begin{equation}
\mat P_{\Xi2N}=\mass_N^{-1}\mat P_{N2\Xi}^T\mass_M\Leftrightarrow
\mass_N\mat P_{\Xi2N}=\mat P_{N2\Xi}^T\mass_M.
\end{equation}

By combining the operators, we can construct projections which directly move the solution from one element to another (in some sense ``hiding'' the mortar) to be
\begin{equation}
\begin{split}
\mat P_{N2M}&=\mat{P}_{\Xi2M}\mat{P}_{N2\Xi},\\
\mat P_{M2N}&=\mat{P}_{\Xi2N}\mat{P}_{M2\Xi}.
\end{split}
\end{equation}
Note, that in this paper we only consider LGL-nodes for the approximating the $L_2$-projection. However, the approach in \cite{Kopriva2002} is more general as it also considers Legendre Gauss nodes. Also, the construction of the projection operators on interfaces with hanging nodes is briefly discussed.

\section{Experimental Order of Convergence - Degree Preserving Element based Finite Difference Operators}\label{sec:App C}

Besides Discontinuous Galerkin SBP operators, we analyze the convergence of degree preserving, element based finite difference operators (DPEBFD) operators. As described in \cite{Friedrich2016} these operators are SBP operators by construction, for which our entropy stable discretization remains stable. The norm matrix of the DPEBFD operator integrates polynomials up to degree of $2p+1$ exactly, where $p$ denotes the minimum polynomial degree of all elements. In comparison to SBP finite difference operators as in \cite{DCDRF2014} these operators are element based, meaning that the number of nodes is fixed as for DG operators.

As we focus on elements with SBP operators of the same degree, we set all elements to be DPEBFD elements with degree $p$ where $p=2,3$. To approximate the convergence order of the non-conforming discretization, we consider the same mesh refinement strategy as in Fig. \ref{mesh3} with element types $A,B,B$. These types are set up in the following way:
\begin{itemize}
\item Element $A$ with $22$ nodes in $x$- and $y$-direction,\\
\item Element $B$ with $24$ nodes in $x$- and $y$-direction,\\
\item Element $C$ with $22$ nodes in $x$- and $y$-direction.\\
\end{itemize}
This leads to a mesh considering $h/p$ refinement. Here, we obtain the results in Tables \ref{DP2}-\ref{DP3}

\begin{table}[ht]
\begin{center}
\textbf{DPEBFD SBP operators}\\
\vspace{0.3cm}
\begin{minipage}{0.4\textwidth}
\begin{center}
\begin{tabular}{c|c|c}
\hline
DOFS & $L_2$ &  EOC\\
\hline
6176   &6.09E-01&\\
24704  &1.60E-01&1.9\\
98816  &2.44E-02&2.7\\
395264 &3.11E-03&3.0\\
1581056&3.97E-04&3.0\\
\hline
\end{tabular}\\[0.1cm]
\caption{Experimental order of convergence for DPEBFD operators with $p=2$.}
\label{DP2}
\end{center}
\end{minipage}
\qquad
\begin{minipage}{0.4\textwidth}
\begin{center}
\begin{tabular}{c|c|c}
\hline
DOFS & $L_2$ &  EOC\\
\hline
768   &3.47E-01&\\
3072  &8.70E-02&2.0\\
12288 &6.83E-03&3.7\\
49152 &4.69E-04&3.9\\
196608&2.99E-05&4.0\\
\hline
\end{tabular}\\[0.1cm]
\caption{Experimental order of convergence for DPEBFD operators with $p=3$.}
\label{DP3}
\end{center}
\end{minipage}
\end{center}
\end{table}

As documented in \cite{Friedrich2016} we numerically verify an EOC of $p+1$. So when considering degree preserving SBP operators our entropy stable non-conforming method can handle $h/p$ refinement and possesses full order. However when considering DG-operators we obtain a smaller $L_2$ error for a more coarse mesh. We do not claim, that DPEBFD operators have the best error properties, but considering these operators is a possible cure for retaining a full order scheme. The development of optimal degree preserving SBP operators is left for future work.

\end{document}